\documentclass[11pt]{amsart}

\usepackage{amsfonts, fullpage, fancyhdr, qtree, amsmath, tipa, amssymb, hyperref, url, amsthm, subfigure, enumitem, setspace, mathrsfs, verbatim, color}

\usepackage[pdftex]{graphicx}
\usepackage[all,cmtip]{xy}

\theoremstyle{norm}
\newtheorem{thm}{Theorem}[section]
\newtheorem{lem}[thm]{Lemma}

\newtheorem{prop}[thm]{Proposition}
\newtheorem*{mainA}{Theorem A}
\newtheorem*{mainB}{Theorem B}
\newtheorem*{mainC}{Theorem C}
\newtheorem*{mainD}{Theorem D}

\newtheorem{cor}[thm]{Corollary}

\theoremstyle{definition}
\newtheorem{rem}[thm]{Remark}
\newtheorem{df}[thm]{Definition}
\newtheorem{exam}[thm]{Example}

\numberwithin{equation}{section}

\newcommand{\K}{\left\lceil\frac{k}{2n}\right\rceil}

\newcommand{\C}{\widetilde{C}_n}

\newcommand{\Cnk}{\mathcal{C}_n^k}

\newcommand{\Cnleqk}{\mathcal{C}_{n-1}^{\leq k}}

\newcommand{\Z}{\mathbb{Z}}
\newcommand{\R}{\mathbb{R}}

\DeclareMathOperator*{\modulo}{mod}

\DeclareMathOperator*{\Aff}{Aff}
\DeclareMathOperator*{\GL}{GL}

\title[Bijective Projections on Parabolic Quotients of Affine Weyl Groups]{Bijective Projections on Parabolic Quotients of Affine Weyl Groups}

\author[E. Beazley, M. Nichols, M. H. Park, X. Shi, and A. Youcis]{Elizabeth Beazley, Margaret Nichols, Min Hae Park, XiaoLin Shi, and Alexander Youcis}\thanks{The authors were partially supported by NSF grant DMS-0850577}

\address{Department of Mathematics, Haverford College, Haverford, PA}
\email{ebeazley@haverford.edu}
\address{Department of Mathematics, University of Chicago, Chicago, IL, USA}
\email{mnichols@math.uchicago.edu}
\address{Department of Mathematics, Williams College, Williamstown, MA}
\email{Min.Hae.Park@williams.edu}
\address{Department of Mathematics, Harvard University, Cambridge, MA, USA}
\email{dannyshi@math.harvard.edu}
\address{Department of Mathematics, University of California, Berkeley, CA, USA}
\email{ayoucis@math.berkeley.edu}

\begin{document}


\begin{abstract}
Affine Weyl groups and their parabolic quotients are used extensively as indexing sets for objects in combinatorics, representation theory, algebraic geometry, and number theory. Moreover, in the classical Lie types we can conveniently realize the elements of these quotients via intuitive geometric and combinatorial models such as abaci, alcoves, coroot lattice points, core partitions, and bounded partitions.  In \cite{BJV} Berg, Jones, and Vazirani described a bijection between $n$-cores with first part equal to $k$ and $(n-1)$-cores with first part less than or equal to $k$, and they interpret this bijection in terms of these other combinatorial models for the quotient of the affine symmetric group by the finite symmetric group.  In this paper we discuss how to generalize the bijection of Berg-Jones-Vazirani to parabolic quotients of affine Weyl groups in type $C$.  We develop techniques using the associated affine hyperplane arrangement to interpret this bijection geometrically as a projection of alcoves onto the hyperplane containing their coroot lattice points.  We are thereby able to analyze this bijective projection in the language of various additional combinatorial models developed by Hanusa and Jones in \cite{HanusaJones}, such as abaci, core partitions, and canonical reduced expressions in the Coxeter group.
\end{abstract}

\maketitle


\section{Introduction}\label{S:intro}

Core partitions initially arose in the study of the representation theory of the symmetric group over a finite field.  In modular representation theory, cores index blocks in the decomposition of the group algebra; see \cite{JamesKerber}.   Cores now appear as indexing sets for many objects in combinatorics, representation theory, algebraic geometry, and number theory. For example, Garvan, Kim, and Stanton combinatorially proved Ramanujan's congruences for the partition function using statistics called cranks, which are closely related to core partitions \cite{GKS}.  In another representation-theoretic context, $n$-cores correspond to extremal vectors in a highest weight crystal for $\widehat{\mathfrak{sl}_n}$; see \cite{MM}.  In algebraic geometry, $n$-cores arise in expansions of the $k$-Schur functions of Lapointe, Lascoux, and Morse \cite{LapMorse05}, which Lam then proved represent the Schubert basis in the homology of the affine Grassmannian \cite{LamkSchur}.  Cores are also related to rational smoothness of Schubert varieties inside the affine Grassmannian as shown by Billey and Mitchell \cite{BilleyMitchell}.

In work related to the study of irreducible Specht modules over the Hecke algebra of the symmetric group  \cite{BV}, Berg and Vazirani prove that there is a bijection between the set $\mathcal{C}_n^k$ of $n$-cores with first part equal to $k$ and the set $\mathcal{C}_{n-1}^{\leq k}$ of $(n-1)$-cores with first part less than or equal to $k$.  Because of the wide array of connections among core partitions and other areas of mathematics, many additional combinatorial models for core partitions have been developed. For example, $n$-cores also index minimal length coset representatives in the quotient $\widetilde{\mathcal{S}}_n/\mathcal{S}_n$ of the affine symmetric group by the finite symmetric group.  There are also interpretations in terms of abacus diagrams, root lattice points, bounded partitions, and certain alcoves in the affine hyperplane arrangement corresponding to the affine symmetric group.  Connections among these models play crucial roles in various areas of mathematics; for example, the connection between cores and bounded partitions was fundamental in the development of the $k$-Schur functions.  In \cite{BJV}, Berg, Jones, and Vazirani interpret the equipotence of $\Cnk$ and $\Cnleqk$ geometrically in terms of the alcove model for $\widetilde{\mathcal{S}}_n/\mathcal{S}_n$, thereby obtaining several additional combinatorial descriptions for this bijection.

Our main theorem is a generalization of the results in \cite{BJV} to Lie type $C$.  In type $C$, this parabolic quotient is known to be in bijection with the set of lecture hall partitions introduced by Bousquet-M\'{e}lou and Eriksson \cite{BME} and certain mirrored $\Z$-permutations defined by Eriksson in his thesis \cite{ErikThesis}.  Hanusa and Jones additionally define bijections from the quotient $\widetilde{C}_n/C_n$ to symmetric core partitions, abacus diagrams, bounded partitions, and canonical reduced expressions in the Coxeter group in \cite{HanusaJones}.  There is also a classical geometric connection to certain hyperplane arrangements through the language of root systems; see \cite{Bourbaki} and \cite{Humphreys}.

The crucial ingredient in many of the results contained in this paper is the ability to provide a geometric interpretation in terms of alcove walks in the affine hyperplane arrangement.  These piecewise linear paths in the real span of the weight lattice were introduced by Littelmann, who calls them Lakshmibai-Seshadri paths \cite{LakSesh}, in order to prove a Littlewood-Richardson rule for decomposing the tensor product of two simple highest weight modules of a complex symmetrizable Kac-Moody algebra into its irreducible components \cite{Littel}.  Alcove walks now arise throughout the literature in representation theory and algebraic geometry, and they often seem in many instances to provide the most natural framework for type-free generalizations of results which had previously only been known in type $A$.  For example, Schwer provides a formula for the Hall-Littlewood polynomials of arbitrary type in terms of alcove walks \cite{Schwer}, and this formula was generalized by Ram and Yip to Macdonald polynomials using similar language \cite{RamYip}.  There is also an explicit correspondence between alcove walks and saturated chains in strong Bruhat order on the affine Weyl group, which gives rise to type-free applications in equivariant $K$-theory of flag varieties \cite{LP1} and the uniform construction of tensor products of certain Kirillov-Reshetikhin crystals \cite{LNSSS}.

\subsection{Summary of the Main Results}\label{sec:summary}



We start by defining a map $\Phi_n$ on elements of the parabolic quotient $\widetilde{C}_n/C_n$.  The map $\Phi_n$ acts on symmetric $(2n)$-cores, which index the minimal length coset representatives of $\widetilde{C}_n/C_n$, as proved in \cite{HanusaJones}.  Given a symmetric $(2n)$-core, the map $\Phi_n$ acts as follows: first, label the boxes of the $(2n)$-core with the elements of $\Z/2n\Z$ repeating along diagonals by labeling box $(i,j)$ with $j - i \pmod{2n}$.  Then, delete all the rows that end with the same element of $\Z/2n\Z$ as the first row.  Finally, delete all the columns that end with the same element of $\Z/2n\Z$ as the first column.  It can be shown that the image of $\Phi_n$ is a set of $(2n-2)$-cores, which correspond to minimal length coset representatives of $\widetilde{C}_{n-1}/C_{n-1}$.

\begin{mainA}[Theorem~\ref{MainTheorem1}] The map $\Phi_n^k$ given by restricting $\Phi_n$ to $\mathscr{S}_{2n}^k$, the set of symmetric $(2n)$-cores with first part equal to $k$, becomes a bijection onto its image $\mathscr{S}_{2n-2}^{\leq k - \lceil \frac{k}{n} \rceil}$, the set of symmetric $(2n-2)$-cores with first part at most $k - \left\lceil \frac{k}{n} \right\rceil$.
\end{mainA}

The proof of Theorem A uses the bijection between symmetric $(2n)$-cores and balanced flush abacus diagrams with $(2n)$-runners, as introduced by Hanusa-Jones \cite{HanusaJones}.  Under this bijection $F_{\mathscr{A}}$, the map $\Phi_n$ induces a map $\mathscr{A}_{2n} \longrightarrow \mathscr{A}_{2n-2}$ on abacus diagrams with $(2n)$ runners to those with $(2n-2)$ runners.   In addition, there is a bijection $F_{\mathscr{R}}$ from abacus diagrams on $(2n)$-runners to lattice points in $\Z^n$.  We are able to explicitly describe the bijections $\Phi_n^k$ and their (co)domains in terms of all of these combinatorial models for the parabolic quotient $\widetilde{C}_n/C_n$ in a manner which makes the diagram below commute.  It turns out that the induced map $\Phi_n$ on abacus diagrams is the most natural to describe.  Theorem A is then proved by translating the condition imposed on symmetric cores to the corresponding abacus diagrams and coroot lattice points. 

\begin{equation*}\label{E:CommDiagram}
 \xymatrix{
\mathscr{S}_{2n} \ar[d]^{\Phi_n} \ar[r]^{F_{\mathscr{A}}} & \mathscr{A}_{2n}    \ar[r]^{F_{\mathscr{R}}} \ar[d]^{\Phi_n} & \mathscr{R}_{2n} \ar[d]^{\Phi_n}  \\ 
\mathscr{S}_{2n-2}  \ar[r]^{F_{\mathscr{A}}}  & \mathscr{A}_{2n-2} \ar[r]^{F_{\mathscr{R}}}  & \mathscr{R}_{2n-2}  
}
\end{equation*}



We also interpret the bijections $\Phi_n^k$ both algebraically and geometrically.  In terms of reduced words in the quotient of the corresponding Coxeter group $\widetilde{C}_n/C_n$, the result we obtain is the following.

\begin{mainB}[Corollaries \ref{MainTheorem2a} and \ref{MainTheorem2b}] For any reduced word in $\widetilde{C}_n/C_n$ corresponding to a symmetric $(2n)$-core with first part equal to $k$, applying $\Phi_n$ decreases the length of the word by exactly $k$.
\end{mainB}

The novelty of Theorem B lies more in the method of its proof, which uses delicate geometric arguments on the associated affine hyperplane arrangement.  In fact, we provide two distinct proofs of Theorem B, one geometric and the other algebraic.  The geometric proof uses the theory of alcove walks to study $\Phi_n$.  The correspondence between the coroot lattice and reduced words in $\widetilde{C}_n/C_n$ gives an induced action of $\Phi_n$ on alcoves in $\R^n$.  Under this correspondence, the lattice points of the alcoves corresponding to elements of $\mathscr{S}_{2n}^k$ all lie on a single hyperplane.  Moreover, when we identify this hyperplane with the Euclidean space $\R^{n-1}$ associated to $\widetilde{C}_{n-1}$, the map $\Phi_n^k$ may be realized as a geometric projection of the alcoves onto the hyperplane containing their coroot lattice points. 
\begin{figure*}
\begin{center}
  \includegraphics[scale=.85]{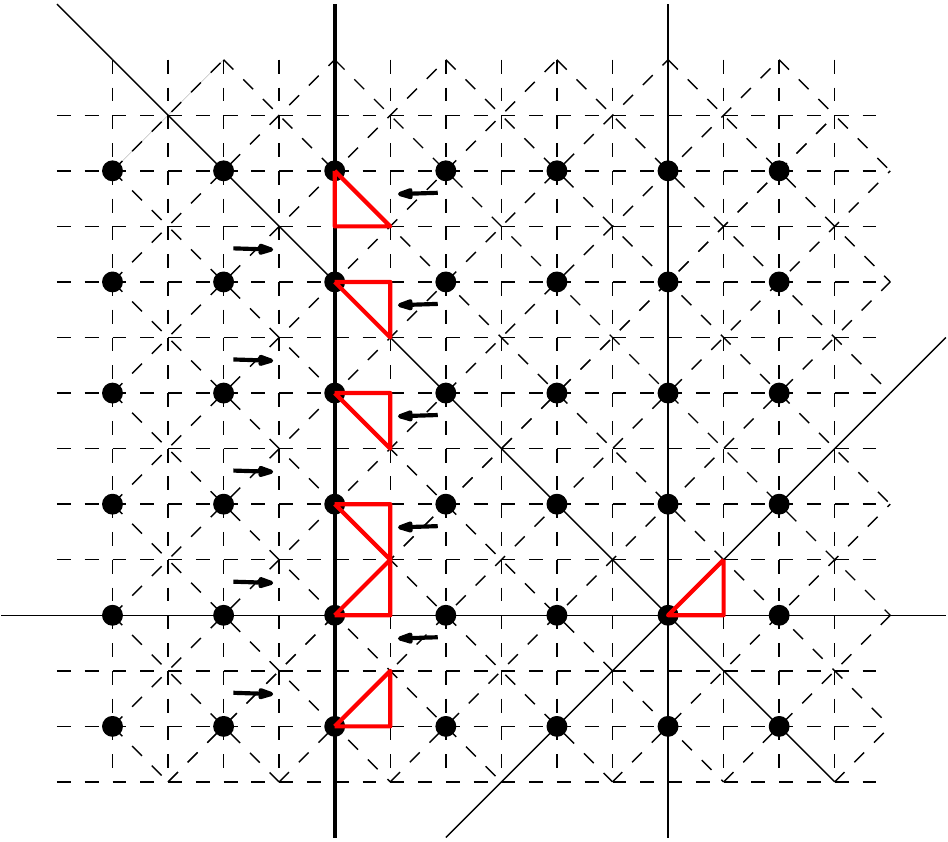}
\includegraphics[scale=.85]{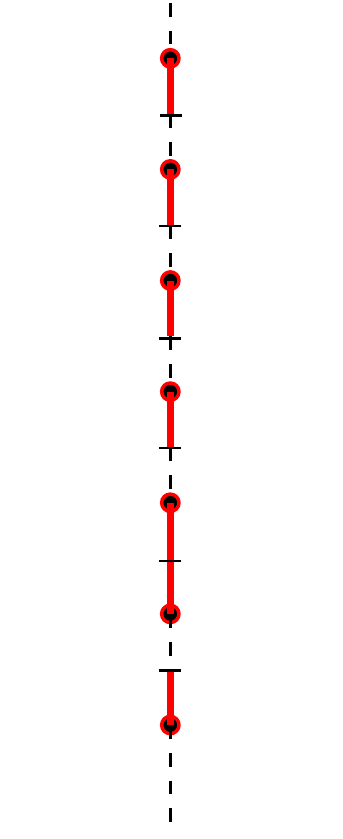}
\caption{A visualization of $\Phi_2^{12}$ as a projection from $\widetilde{C}_2/C_2$ to $\widetilde{C}_1/C_1$.}
\label{fig:projection}       
\end{center}
\end{figure*}
 In particular, using the correspondence between reduced words and minimal length alcove walks, we obtain the following.

\begin{mainC}[Theorem~\ref{MainTheorem3}] Let $w$ be a minimal length coset representative for $\widetilde{C}_n/C_n$ such that the symmetric core partition corresponding to $w$ has first part equal to $k$.  If 
$$\mathcal{A}^1 \to \cdots \to \mathcal{A}^r $$
is an alcove walk for $w$, then 
\begin{equation}\label{Main3}
\pi(\mathcal{A}^1) \to \cdots \to \pi(\mathcal{A}^r) 
\end{equation}
is an alcove walk for $\Phi_n(w)$.  Here, $\pi$ is the projection onto the hyperplane containing the coroot lattice points of the symmetric $(2n)$-core partitions with first part $k$.  Moreover, if one removes all repetitions of the alcoves in~\eqref{Main3}, the result is a minimal length alcove walk for $\Phi_n(w)$.
\end{mainC}

The geometric proof of Theorem B relies upon the interpretation of the projection in terms of alcove walks articulated in Theorem C.  We remark that by considering the root system of type $A$ instead of type $C$, our proof of Theorem C can be modified to yield a similar statement for the projection map in type $A$, which sends $n$-cores with first part equal to $k$ to $(n-1)$-cores with first part less than or equal to $k$.  The algebraic proof of Theorem B presents an explicit description of $\Phi_n$ on reduced words.  Given a minimal length coset representative $w \in \widetilde{C}_n/C_n$, there is an action of the reduced word on the corresponding abacus diagram.  We analyze this action on abaci and provide an explicit algorithm that constructs a reduced word for $\Phi_n(w)$ in $\ell(w)$-steps.


The map $\Phi_n$ also exhibits other nice combinatorial properties which suggest applications to other areas of mathematics, particularly in the direction of affine Schubert calculus.  For example, we prove the following theorem, showing that $\Phi_n$ preserves strong Bruhat order.

\begin{mainD}[Theorem~\ref{BruhatPreserved}] Fix two positive integers $n$ and $k$.  Given two elements $x$ and $y$ in $\C/C_n$ whose associated coroot lattice points lie in the domain of $\Phi_n^k$, then $x \geq_B y$ if and only if $\pi(x) \geq_B \pi(y)$. 
\end{mainD}

\noindent To prove Theorem D, we use the equivalence between the containment of core partitions and domination in the strong Bruhat order, which was introduced in section 5.3 of Hanusa-Jones \cite{HanusaJones} in order to answer a question of Billey and Mitchell \cite{BilleyMitchell}.

\subsection{Applications and Directions for Future Work}\label{sec:futurework}

The authors strongly suspect that analogous results can be achieved in Lie types $B$ and $D$, and this paper provides a solid framework for making such generalizations, although no precise statements have yet been formulated.  One primary difference in other Lie types is that the core partitions of Hanusa-Jones have dynamic labelings, which makes the bijection on core partitions difficult to conjecture.  Therefore, in types $B$ and $D$ the geometry of the alcove walk model developed in this paper will be crucial, not only for proving results, but also for formulating conjectures.  Another additional difficulty with a direct generalization from type $C$ is that the domains for the projections provably do not lie on any of the root hyperplanes themselves.  This was particularly surprising for the authors to discover about type $B$, since combinatorially type $B$ is a subset of type $C$, and geometrically they are dual.  However, much of the groundwork for a geometric analysis of types $B$ and $D$ has been laid in Section \ref{sectionGeometric}.

There are also directions for future work on the combinatorics of many related partially ordered sets.  As Theorem D suggests, a great deal of combinatorial structure is preserved by the bijective projection developed in this paper.  In addition to strong Bruhat order, there are other natural partial orderings to consider on these parabolic quotients.  For example, it would be interesting to more explicitly describe which intervals comprise the domains and images of these bijections.  

Although the statements of many of the results of this paper are combinatorial in nature, both the motivation for the work and the most promising future directions are either algebro-geometric or representation-theoretic.  The homology of the affine Grassmannian has a basis of Schubert classes which are indexed by elements of the parabolic quotients studied in this paper.  The fact that the projection map preserves strong Bruhat order means that Schubert cells in one dimension map to Schubert cells one dimension lower.  Therefore, the bijective projections developed in \cite{BJV} and this paper may yield a means by which one can construct inductive proofs in affine Schubert calculus.  In \cite{RamAlcove}, Ram discusses how the root operators $\tilde{e_i}$ and $\tilde{f_i}$ coincide with the rank one crystal base operators after projection onto the line orthogonal to a hyperplane determined by the simple root $\alpha_i$.  It would be interesting to see if the projections defined in this paper can be similarly interpreted into the language of affine crystals.

The geometric results on alcove walks in Section \ref{sectionGeometric} in particular may have other potential applications in algebraic geometry and representation theory.  In the study of Shimura varieties with Iwahori level structure, Haines and Ng\^{o} develop the notion of an alcove walk in the $w$-direction \cite{HaiNgo}.  The alcove walk algebra, developed by Ram in \cite{RamAlcove} as a refinement of the Littelmann path model of \cite{Littel}, provides a combinatorial method for working with the affine Hecke algebra. This model had already been used by Schwer to provide a combinatorial description of the Hall-Littlewood polynomials \cite{Schwer}, and Ram and Yip further apply the alcove walk algebra to the theory of Macdonald polynomials in \cite{RamYip}.  Parkinson, Ram, and Schwer present a refinement of Ram's alcove walk model in order to study analogs of Mirkovi\'{c}-Vilonen cycles in the affine flag variety \cite{PRS}.  Essential to the work on alcove walks in these various contexts is the ability to additionally track information about the direction or orientation of various parts of the walk.  Potential applications of the work in this paper to Shimura varieties, Mirkovi\'{c}-Vilonen cycles, Macdonald polynomials, and the affine Hecke algebra therefore arise from the refined information obtained in Section \ref{sectionGeometric} on so-called perpendicular and parallel steps in alcove walks.

\subsection{Organization of the Paper}\label{sec:organization}

In Section \ref{sectionWeylGroups}, we provide a brief review of the language of root systems, Weyl groups and their parabolic quotients, and the affine hyperplane arrangement which gives rise to the alcove model for affine Weyl groups.  In Section \ref{sectionCombinatorialModels}, we review the combinatorial models for the minimal length coset representatives of the parabolic quotient $\widetilde{W}/W$ of the affine Weyl group by the finite Weyl group, including an overview of the relevant combinatorics in affine type $C$ from Hanusa-Jones \cite{HanusaJones}.  We summarize in Section \ref{sectionReview} the results obtained in \cite{BJV} by Berg, Jones, and Vazirani for the case of $W = S_n$.

The new results begin in Sections \ref{MapOnCores}-\ref{DomainCodomain}, where we develop the map $\Phi_n$ on the various combinatorial models for the type $C$ quotient, proving Theorem A in Section \ref{MapCoroot}.  Section~\ref{sectionGeometric} develops the geometry of $\Phi_n$ in terms of alcove walks in the affine hyperplane arrangement.  In particular, we extend the result obtained in Section~\ref{MapCoroot} which says that the domains of $\Phi_n$ may be partitioned into hyperplanes, and show that when we identify these hyperplanes with $\R^{n-1}$, then $\Phi_n$ may be regarded as a projection from alcoves of $\widetilde{C}_n/C_n$ in $\R^n$ to alcoves of $\widetilde{C}_{n-1}/C_{n-1}$ in $\R^{n-1}$.  This geometric interpretation provides a constructive proof for Theorem C.  In Section~\ref{sectionAction}, we provide an explicit algorithm that determines the action of $\Phi_n$ on reduced words of $\widetilde{C}_n/C_n$.  The geometry developed in Section~\ref{sectionGeometric} and the algorithm provided in Section~\ref{sectionAction} give two distinct proofs of Theorem B.  Finally, in Section~\ref{sectionProperties}, we prove Theorem D, showing that $\Phi_n$ preserves the strong Bruhat order on each of its hyperplane domains.

\subsection{Acknowledgements}

The results in this paper were obtained during the SMALL Research Experience for Undergraduates at Williams College in Summer 2012.  The authors wish to thank the National Science Foundation for its financial support and Williams College for its generous additional support and excellent working conditions.  The authors would also like to thank the anonymous referees for their detailed readings and many constructive suggestions which greatly improved the exposition of this paper.


\section{Weyl Groups, Root Systems, and Alcoves} \label{sectionWeylGroups}


We begin by establishing some notation and providing a brief review of the necessary background on Coxeter groups, Weyl groups, and their associated root systems.  We refer the reader to \cite{BjornerBrenti} and \cite{Humphreys} for more detail on any of the definitions and results stated in this section.

Let $(W,S)$ denote a Coxeter system with a set of generators $S$ for the Coxeter group $W$.  Since the generators $s \in S$ all have order 2, each $w \neq 1$ in $W$ can be written as $w = s_1s_2 \cdots s_r$ for some $s_i \in S$.  The minimal number of generators required to write $w$ as a product is called the length of $w$, which we denote by $\ell(w)$.  The Coxeter group $W$ is partially ordered by strong Bruhat order, which we denote by $\leq$. If the order of the product of any two generators in $S$ equals either $2, 3, 4,$ or $6$, then the corresponding Coxeter group $W$ is a \emph{Weyl group}.  The Weyl groups of interest in this paper are $A_n$ and $C_n$, as well as their affine analogs.  The Dynkin diagrams for types $A_n$ and $C_n$ are shown in Figure~\ref{DynkinDiagramsABCD}.

Let $\Phi$ denote the root system for the Weyl group $W$, which has a basis of simple roots $\Delta$.  We remark that in all cases of interest in this paper, the root system $\Phi$ is irreducible.   An element of $\Phi$ which is a non-negative integral combination of simple roots is positive, and $\Phi$ is the disjoint union of positive and negative roots $\Phi = \Phi^+ \cup \Phi^-$.   Since $\Phi$ is irreducible, there is a unique highest root $\tilde{\alpha}$ which satisfies that $\tilde{\alpha} - \alpha$ is a sum of simple roots for all $\alpha \in \Phi^+$.  For example, if $\varepsilon_1, \ldots, \varepsilon_n$ denotes the standard basis of $\mathbb{R}^n$, then the positive roots of the Weyl group $C_n$ are $2 \varepsilon_i$ for $1 \leq i \leq n$, and $\varepsilon_i \pm \varepsilon_j$ for $1 \leq i < j \leq n$. The basis of simple roots $\Delta$ consists of $\alpha_1  = \varepsilon_1 - \varepsilon_2, \alpha_2 = \varepsilon_2 - \varepsilon_3, \ldots, \alpha_{n-1} = \varepsilon_{n-1} - \varepsilon_n$, and $\alpha_n = 2 \varepsilon_n$, and the unique long root is $\widetilde{\alpha} = 2 \varepsilon_1$. 

For any $\alpha \in \Phi$, denote by $s_{\alpha}$  the reflection across the hyperplane perpendicular to $\alpha$ that passes through the origin.  The Weyl group $W$ is then generated by the reflections $s_{\alpha}$ for $\alpha \in \Phi$, and so we may view $W$ as a reflection group acting on the Euclidean space $V = \R^n$.  Conversely, given a generator $s \in S$, by $\alpha_s$ we mean the positive simple root normal to the hyperplane through which $s$ reflects.

Denote by $\Phi^{\vee} := \{ \alpha^{\vee} \, | \, \alpha \in \Phi\}$ the dual root system, where we define the coroots as $\alpha^{\vee} := 2 \alpha / \langle \alpha, \alpha \rangle$.  The associated simple coroots are then given by $\Delta^{\vee} := \{ \alpha^{\vee} \, | \, \alpha \in \Delta\}$.  Denote by $\Lambda_R$ the root lattice, which is the $\mathbb{Z}$-span $\Lambda_R$ of $\Phi$ in $V$.  Similarly, denote the coroot lattice by $\Lambda_R^{\vee}$.  



The \emph{affine group} $\Aff(V)$ is the group consisting of all affine reflections across all hyperplanes in $V$.  It is shown in Section 4.1 of \cite{Humphreys} that $\Aff(V)$ is the semidirect product of $\GL(V)$ and the group of translations of elements of $V$.  For each root $\alpha \in \Phi$ and each integer $k$, consider the affine hyperplane 
$$H_{\alpha, k} : = \{ \lambda \in V \, | \, ( \lambda, \alpha ) = k \}. $$
The corresponding affine reflection across this hyperplane is given by
$$s_{\alpha, k} (\lambda):= \lambda - ( ( \lambda, \alpha ) - k ) \alpha^{\vee}. $$
Note that $H_{\alpha, k} = H_{- \alpha, -k}$ and that $s_{\alpha, 0} = s_{\alpha}$.  We occasionally abbreviate $H_{\alpha, 0} = H_{\alpha}$.  Let $\mathcal{H}$ be the collection of all hyperplanes $$\mathcal{H} := \{ H_{\alpha, k} \mid \alpha \in \Phi, k \in \mathbb{Z}\}.$$ The elements of $\mathcal{H}$ are permuted naturally by elements of $W$ and translations in $\Aff(V)$ by elements of the coroot lattice $\Lambda^{\vee}_R$.

Define the \emph{affine Weyl group} $\widetilde{W}$ to be the subgroup of $\Aff(V)$ generated by all affine reflections $s_{\alpha, k}$, where $\alpha \in \Phi$ and $k \in \mathbb{Z}$.  It is well-known that $\widetilde{W}$ is the semidirect product of $W$ and the translation group corresponding to the coroot lattice $\Lambda_R^{\vee}$; see Proposition 4.2 in \cite{Humphreys}.  Since the elements of $\widetilde{W}$ permute the hyperplanes in $\mathcal{H}$, they permute the collection $\mathfrak{A}$ of connected components of $V^{\circ} := V \backslash \cup_{H \in \mathcal{H}} H$.  Each element of $\mathfrak{A}$ is called an \emph{alcove}.  The \emph{fundamental alcove} $\mathcal{A}_{\circ}$ is the alcove $$\mathcal{A}_{\circ} = \{\lambda \in V \, | \, 0 < (\lambda, \alpha) < 1 \text{ for all } \alpha \in \Phi^+ \}.$$
In fact, any alcove $\mathcal{A} \in \mathfrak{A}$ consists of all $\lambda \in V$ satisfying the strict inequalities $k_{\alpha} < (\lambda, \alpha) < k_{\alpha} + 1$, where $\alpha$ runs through $\Phi^+$ and $k_{\alpha} \in \mathbb{Z}$ is some fixed integer depending on $\alpha$.   The \emph{walls} of $\mathcal{A}_{\circ}$ are the hyperplanes $H_{\alpha}$ for $\alpha \in \Delta$, together with $H_{\tilde{\alpha}, 1}$.

Define $\widetilde{S}$ to be the set of reflections
$$\widetilde{S} := \{ s_{\alpha} \mid \alpha \in \Delta \} \cup \{s_{\tilde{\alpha}, 1}\}. $$  We will often relabel $s_0 = s_{\tilde{\alpha}, 1}$ for convenience.   The pair $(\widetilde{W}, \widetilde{S})$ is a Coxeter system (see Proposition 4.3 in \cite{Humphreys}), and the Dynkin diagram for $\widetilde{W}$ in type $C$ appears in Figure~\ref{DynkinDiagramsABCD}.  In order to distinguish among the different Lie types, we will use the notation $\widetilde{A}_n$ and $\widetilde{C}_n$ to denote the affine Weyl groups of rank $n$ in Lie types $A$ and $C$, respectively, and we will occasionally denote by $A_n$ and $C_n$ the corresponding finite Weyl groups.

 \begin{figure}[htp]
\begin{center}
\subfigure[$A_n$ ($n \geq 1$)]{\includegraphics[scale = 0.5]{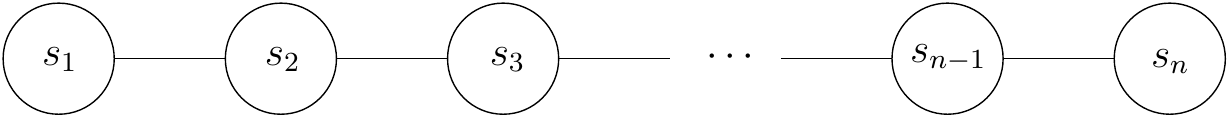}} \hspace{0.5in}
\subfigure[$C_n$ ($n \geq 2$)]{\includegraphics[scale = 0.5]{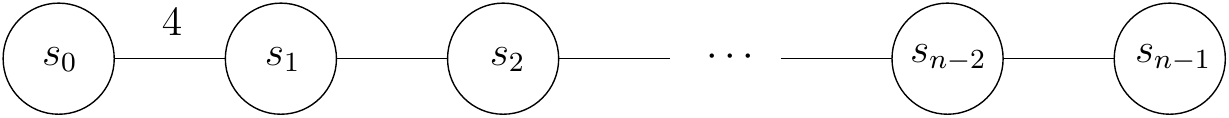}} \hspace{0.5in}
\subfigure[$\widetilde{C}_n$ ($n \geq 2$)]{\includegraphics[scale = 0.5]{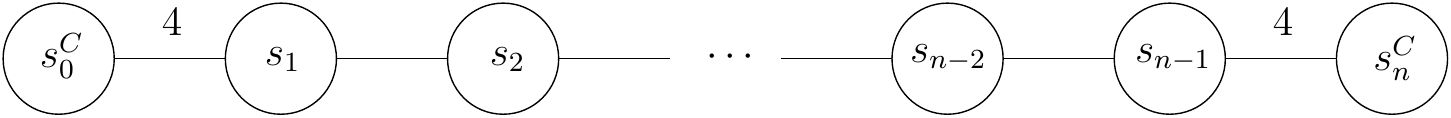}} \hspace{0.5in}
\caption{Dynkin diagrams for the Weyl groups $A_n, C_n,$ and $\widetilde{C}_n$. }
\label{DynkinDiagramsABCD}
\end{center}
\end{figure} 

 The affine Weyl group $\widetilde{W}$ acts transitively on $\mathfrak{A}$; we will consider the action by left multiplication in this paper.  More specifically, for an alcove $\mathcal{A} \in \mathfrak{A}$, the action $s\mathcal{A}$ applies the reflection corresponding to the generator $s$ to the alcove $\mathcal{A}$.  If $w= s_{i_1}s_{i_2} \cdots s_{i_k}$, then $w\mathcal{A}$ is the alcove obtained by applying each of the reflections $s_{i_k}, \dots, s_{i_1}$ to the alcove $\mathcal{A}$ one after the other.  We will often identify the element $w \in \widetilde{W}$ with the alcove $w\mathcal{A}_{\circ} \in \mathfrak{A}$.  The alcoves which correspond to elements in the finite Weyl group $W$ are precisely those adjacent to the origin.  The union of these alcoves is called the \emph{fundamental region}.

\begin{figure}[htp]
\begin{center}
\subfigure[$A_{\circ}$]{\includegraphics[scale = 0.5]{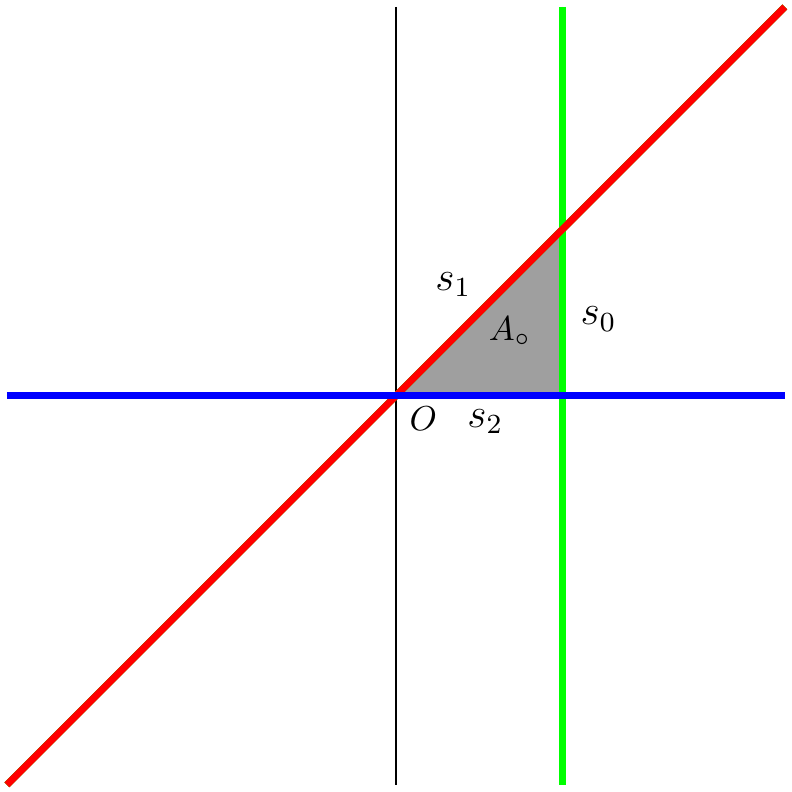}} \hspace{0.3in}
\subfigure[$s_0A_{\circ}$]{\includegraphics[scale = 0.5]{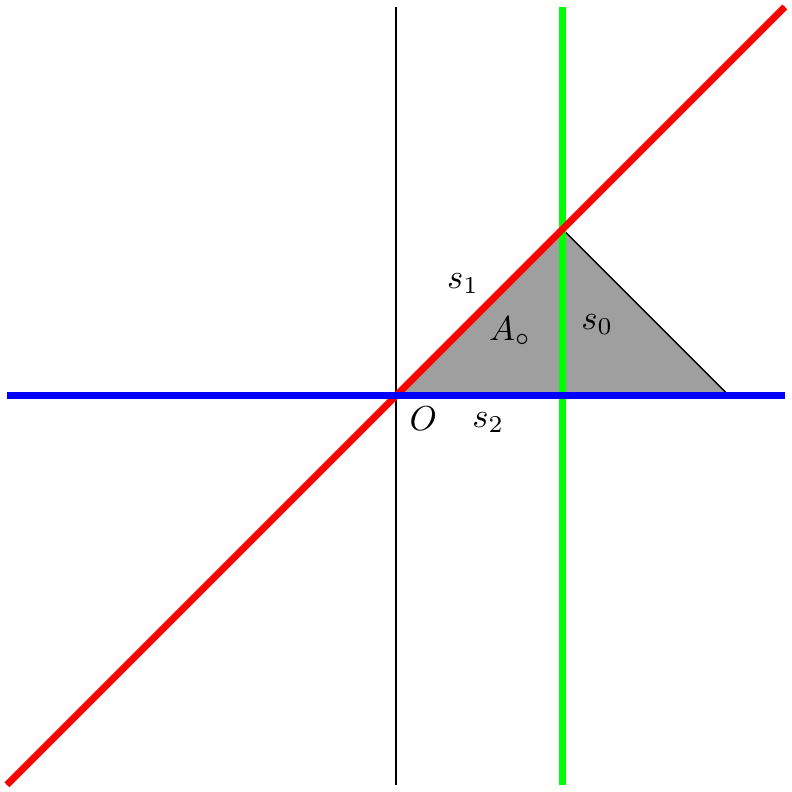}} \hspace{0.3in}
\subfigure[$s_1s_0A_{\circ}$]{\includegraphics[scale = 0.5]{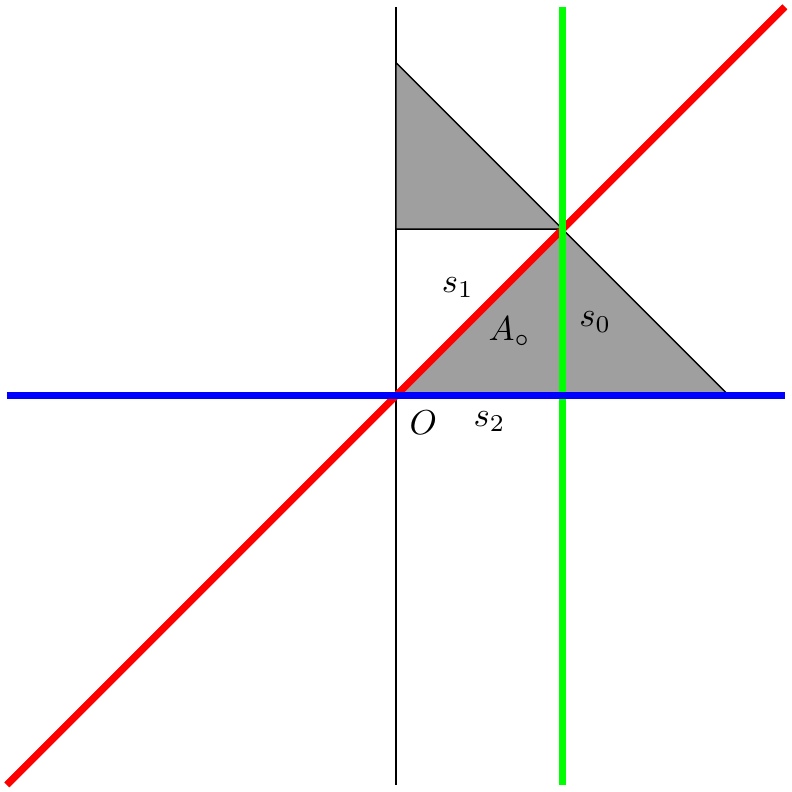}} \hspace{0.3in}
\subfigure[$s_2s_1s_0A_{\circ}$]{\includegraphics[scale = 0.5]{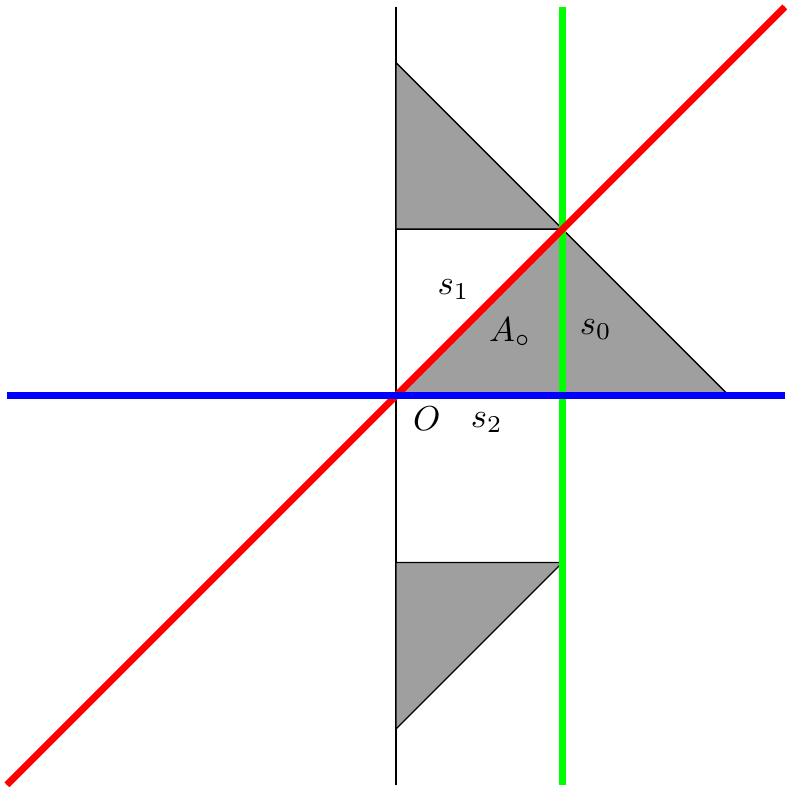}} \hspace{0.6in}
\subfigure[All alcoves]{\includegraphics[scale = 0.5]{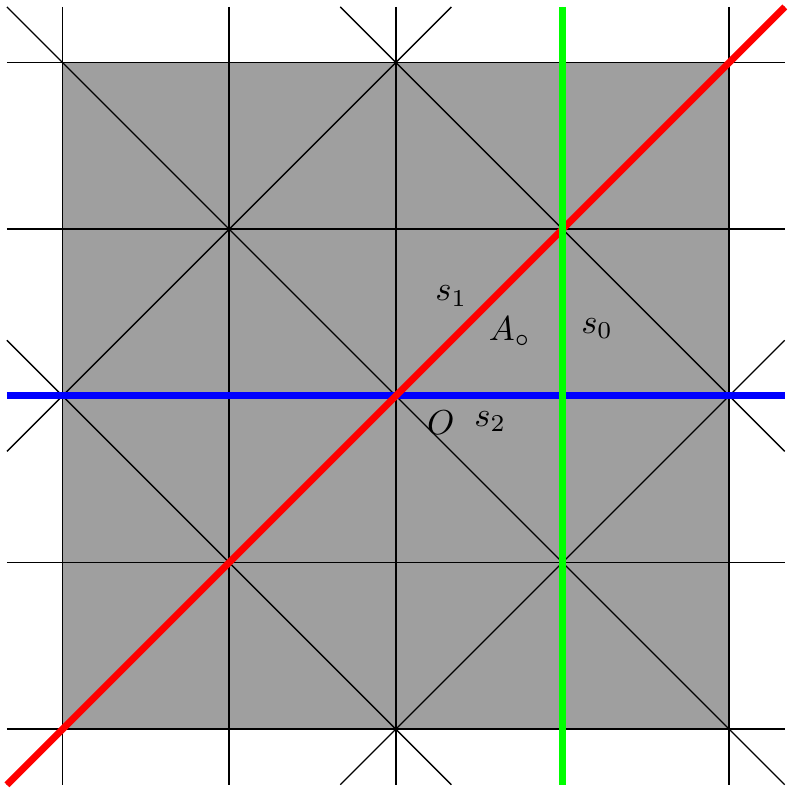}}
\caption{Elements of $\widetilde{C}_2$ permute $A_{\circ}$ transitively. }
\label{fig:A0transitive}
\end{center}
\end{figure}

\begin{exam} \rm
The Weyl group $C_2$ has the presentation 
$$C_2 = \langle s_1, s_2 \mid (s_1s_2)^4 = 1 \rangle, $$
and the affine Weyl group $\widetilde{C}_2$ has the presentation 
\begin{eqnarray}\label{relationsC2}
\widetilde{C}_2 &=& \langle s_0, s_1, s_2 \mid (s_0s_1)^4 = (s_0s_2)^2 = (s_1s_2)^4 = 1 \rangle.
\end{eqnarray}
In Figure~\ref{fig:A0transitive}, the walls of the fundamental alcove $\mathcal{A}_{\circ}$ are labeled with the generators of $\widetilde{C}_2$.  The sequence of pictures demonstrates how the reflections generated by $s_0, s_1, s_2$ permute $\mathcal{A}_{\circ}$ transitively, in addition to the correspondence between words in $\widetilde{W}$ and alcoves in $\R^2$.  
\end{exam}

Given a hyperplane $H = H_{\alpha, k} \in \mathcal{H}$, each alcove lies in one of the two half-spaces defined by $H$.  We say that $H$ \emph{separates} two alcoves $\mathcal{A}$ and $\mathcal{A}'$ if these alcoves lie in different half-spaces relative to $H$.  For example, the hyperplane $H_{\alpha_s,0}$ separates $\mathcal{A}_{\circ}$ and $s\mathcal{A}_{\circ}$ for all $s \in \widetilde{S}$.  Given any alcove $w\mathcal{A}_{\circ} \in \mathfrak{A}$, it is well-known that the length $\ell(w)$ equals the total number of hyperplanes which separate $\mathcal{A}_{\circ}$ and $w\mathcal{A}_{\circ}$; see Section 4.5 in \cite{Humphreys}.

\begin{df}
An \emph{alcove walk} from $\mathcal{A}_{\circ}$ to $w\mathcal{A}_{\circ}$ is a path connecting a point in the interior of $\mathcal{A}_{\circ}$ to a point in the interior of $w\mathcal{A}_{\circ}$, with another condition that the path cannot pass through a vertex of any alcove.  A \textit{step} in an alcove walk is the result of applying a reflection to an alcove across a hyperplane which bounds the alcove. If we consider all such paths from $\mathcal{A}_{\circ}$ to $w\mathcal{A}_{\circ}$, then the minimum number of hyperplanes in $\mathcal{H}$ that such path intersects equals $\ell(w)$.  An alcove walk which crosses the minimum number of hyperplanes is called a \emph{minimal (length) alcove walk} from $\mathcal{A}_{\circ}$ to $w\mathcal{A}_{\circ}$; note that it is not necessarily unique, just as reduced expressions for $w$ are not unique. 
\end{df}

It is a general fact (see \cite{RamAlcove}, for example) that we may label the hyperplanes in $\mathcal{H}$ with generators in $\widetilde{S}$ such that if an alcove walk from $\mathcal{A}_{\circ}$ to $\mathcal{A}' \in \mathfrak{A}$ passes through the hyperplanes labeled $s_{i_1}, s_{i_2}, \ldots, s_{i_k}$, in that order, then the alcove $\mathcal{A}'$ is the alcove $w\mathcal{A}_{\circ}$, where $w= s_{i_1}s_{i_2} \cdots s_{i_k}$.  

\begin{figure}[htp]
\begin{center}
\subfigure[$s_2s_1s_2s_0s_1s_0$]{\includegraphics[scale = 0.6]{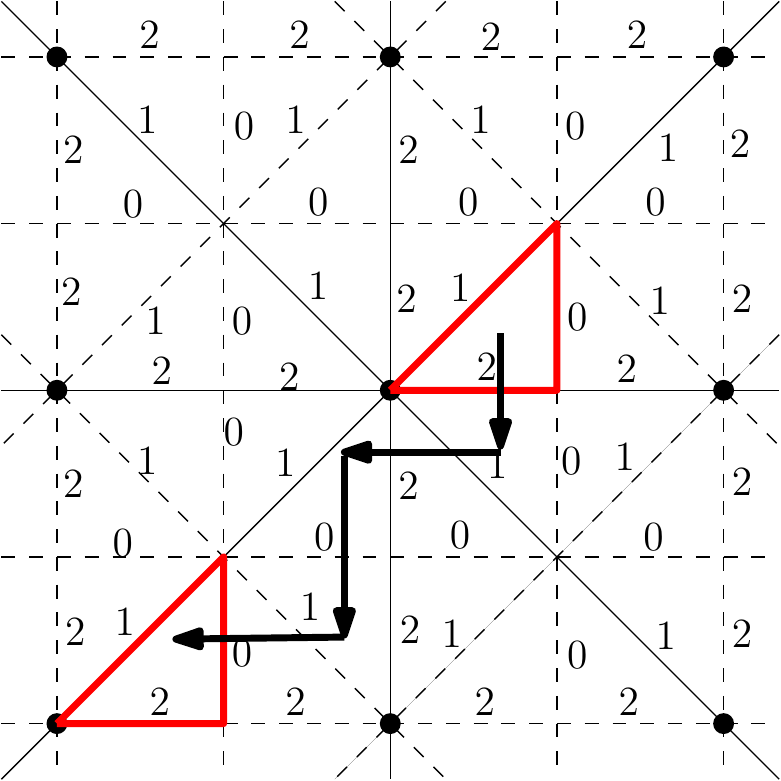}} \hspace{0.5in}
\subfigure[$s_1s_2s_0s_1s_0s_1s_0s_1s_2s_1s_0s_1$]{\includegraphics[scale = 0.6]{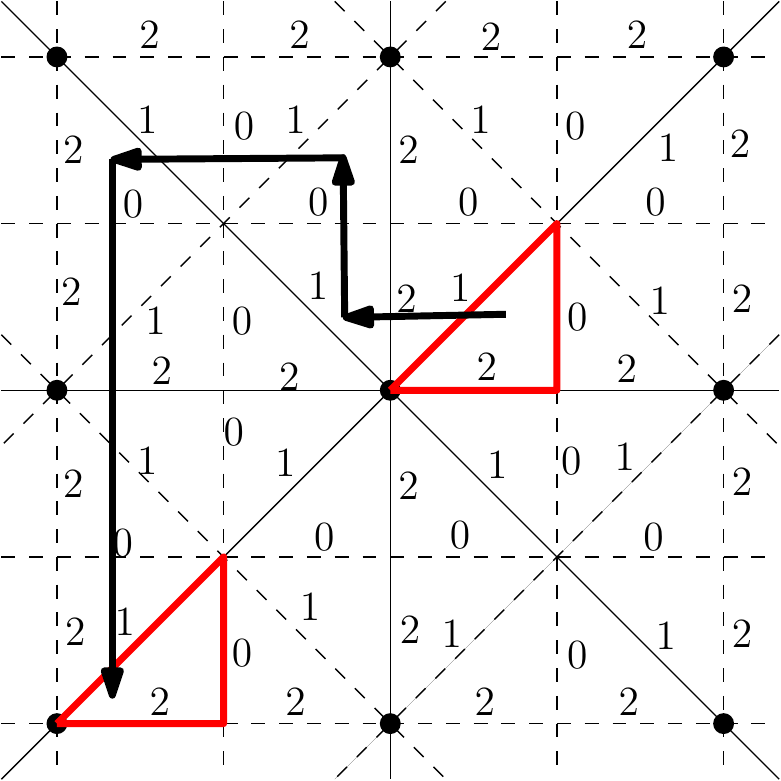}} 
\caption{Examples of alcove walks. }
\label{fig:minAlcoveWalk}
\end{center}
\end{figure}

\begin{exam} \rm
Consider the affine Weyl group $\widetilde{C}_2$. The fundamental alcove $\mathcal{A}_{\circ}$ is highlighted in red in the center, and each of the three walls of $\mathcal{A}_{\circ}$ are labeled by one of the three generators in $\widetilde{S}$, according to whether the neighboring alcove is $s_0\mathcal{A}_{\circ}, s_1\mathcal{A}_{\circ},$ or $s_2\mathcal{A}_{\circ}$, as in Figure~\ref{fig:A0transitive}; here we record only the subscripts for brevity.  Around the origin is a copy of finite $C_2$, so every wall of an alcove in this fundamental region is labeled by one of the two generators in $S$, with the third side labeled by the affine reflection $s_0$.  Continuing this pattern, each of the three walls of every alcove are labeled with a distinct generator.

Figure~\ref{fig:minAlcoveWalk} shows two alcove walks beginning at $\mathcal{A}_{\circ}$ and ending at the same alcove.  The first alcove walk is a minimal alcove walk, and the word corresponding to it is $w = s_2s_1s_2s_0s_1s_0$.  The second alcove walk is not a minimal alcove walk.  The word corresponding to the second alcove walk is $w' = s_1s_2s_0s_1s_0s_1s_0s_1s_2s_1s_0s_1$.  The two words $w$ and $w'$ are equivalent by using the defining relations of $\widetilde{C}_2$ from~\eqref{relationsC2}, as the following calculation demonstrates:
\begin{eqnarray}
s_1s_2\underline{s_0s_1s_0s_1s_0s_1}s_2s_1s_0s_1 &=& s_1s_2s_1\underline{s_0s_2}s_1s_0s_1  \nonumber \\  
&=& \underline{s_1s_2s_1s_2}s_0s_1s_0s_1 \nonumber \\ 
&=& s_2s_1s_2s_1 \underline{s_0s_1s_0s_1} \nonumber \\ 
&=& s_2s_1s_2 \underline{s_1 s_1}s_0s_1s_0 \nonumber \\ 
&=& s_2s_1s_2s_0s_1s_0. \nonumber
\end{eqnarray}
\end{exam}

\section{Combinatorial Models For the Parabolic Quotient $\widetilde{W}/W$}\label{sectionCombinatorialModels}

Having established the geometric interpretation of the affine Weyl group $\widetilde{W}$ in terms of alcoves, we move on to consider the parabolic quotient $\widetilde{W}/W$, the object of our interest in this paper.  It is a well-known fact that each coset in $\widetilde{W}/W$ has a unique minimal length representative element; see Section 2.4 in \cite{BjornerBrenti}.  To index cosets in the quotient, we consider the minimal length element in each coset.  In this section, we introduce three closely related combinatorial models that index minimal length coset representatives of the parabolic quotient.  They are the root lattice point model, abacus diagrams, and core partitions.  

\subsection{Root Lattice Point Model}\label{rootlatticeptmodel}

To start, recall that $\widetilde{W}$ is the semi-direct product of $W$ and the coroot lattice points $\Lambda_R^{\vee}$; that is, $\widetilde{W}=  \Lambda_R^{\vee} \rtimes W$.  We may also identify the Euclidean space $V = \R^n$ with $ \Lambda_R^{\vee} \otimes_{\Z} \R$.  Given an element $w \in \widetilde{W}$, we may associate a coroot lattice point to $w$ by acting on the origin $0 \in V$ by $w$.  Since elements in the finite Weyl group $W$ leave $0 \in V$ unchanged, two elements in the same coset of $\widetilde{W}/W$ send $0 \in V$ to the same coroot lattice point.  Hence, there is a correspondence between coroot lattice points and cosets of $\widetilde{W}/W$.

\begin{rem}\label{rem:coroot}
It should be pointed out that the authors of \cite{BJV} and \cite{HanusaJones} have used the term ``root lattice point'' where we instead say ``coroot lattice point."  This difference in terminology arises because in type $A$, which was studied by the authors of \cite{BJV}, the root and coroot lattices coincide.  In type $C$, however, we must use the coroot lattice specifically, on which there is a well-defined action of $\widetilde{C}_n$.  To preserve the terminology in the existing literature, the authors have decided to refer to this model as the ``root lattice point model."
\end{rem}

\begin{figure}[hpt]
\begin{center}
{\includegraphics[scale = 0.75]{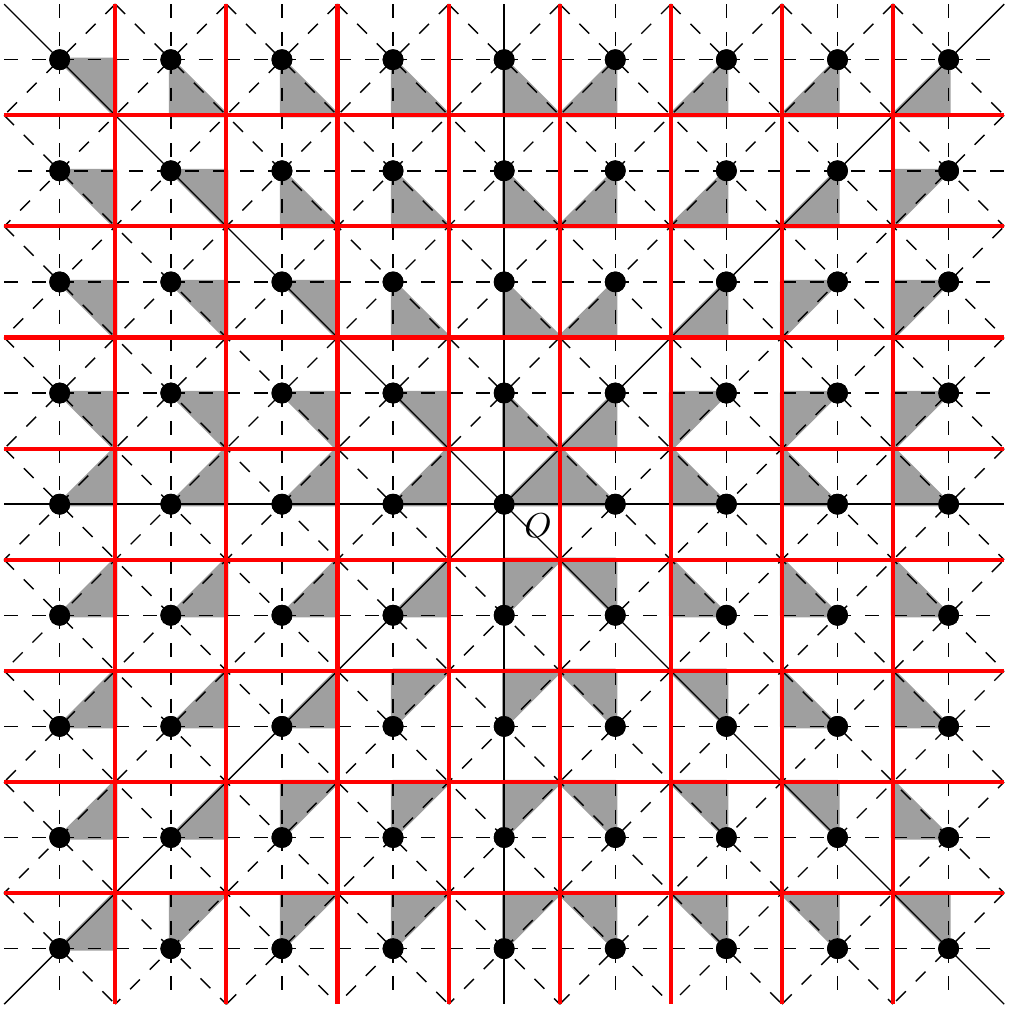}}
\caption{Minimal length coset representatives for $\widetilde{C}_2/C_2$.   }
\label{fig:minCosetRep}
\end{center}
\end{figure}

Geometrically, if $\lambda^{\vee}$ is the coroot lattice point corresponding to a coset in $\widetilde{W}/W$, then the union of the alcoves that represent elements in that coset is a translation of the fundamental region by the coroot $\lambda^{\vee}$.  The minimal length coset representative corresponding to a coroot lattice point is the alcove which has the minimal number of hyperplanes separating it from the fundamental alcove $\mathcal{A}_{\circ}$.  
\begin{df}
A \emph{distinguished alcove} is an alcove which is of minimal length in its coset in $\widetilde{W}/W$.  
\end{df}
Figure~\ref{fig:minCosetRep} shows the distinguished alcoves, or minimal length coset representatives, for $\widetilde{C}_2/C_2$ shaded in gray.  The coroot lattice is highlighted with bold black dots, and each translate of the fundamental region by a coroot is outlined in solid gray.

\begin{df}\label{shiftedWeyl}
The complement of the collection of the hyperplanes $\{ H_{\alpha} \mid \alpha \in \Phi^+\}$ partitions $\R^n$ into \emph{Weyl chambers}, which are highlighted in red in Figure~\ref{fig:minCosetRep}.  Now consider the complement of the collection $\{ H_{\alpha, 1} \mid \alpha \in \Phi^+\}$, shown in blue in Figure~\ref{fig:minCosetRep}.  Note that these hyperplanes do not all meet in a single point, but that only $|W|$ of these regions actually contain any coroot lattice points.  We refer to these $|W|$ regions as the \emph{shifted Weyl chambers}.
\end{df}
The key observation about these shifted Weyl chambers is that, within each shifted Weyl chamber, all distinguished alcoves are simply translates by an element of the coroot lattice of the same alcove in the fundamental region.  We shall exploit this symmetry later in Section \ref{sectionGeometric}.

\subsection{Abacus Diagrams}\label{abacusmodel}

The development of abacus diagrams typically requires the concept of mirrored $\mathbb{Z}$-permutations.  We provide only a brief discussion of mirrored $\Z$-permutations in order to move directly to abacus diagrams; the interested reader is referred to Sections 2 and 3 in \cite{HanusaJones} for more details.

\begin{df}\label{mirroredPermutation} \rm (Definition 2.1 in \cite{HanusaJones})
Fix a positive integer $n$ and let $N = 2n+1$.  A bijection $w: \mathbb{Z} \to \mathbb{Z}$ is a \emph{mirrored $\mathbb{Z}$-permutation} if $w(i+N) = w(i) + N$ and $w(-i) = -w(i)$ for all $i \in \mathbb{Z}$.  
\end{df}

It is easy to see from the equations in Definition~\ref{mirroredPermutation} that a mirrored $\mathbb{Z}$-permutation $w$ is completely determined by its action on $\{ 1, 2, \ldots, n \}$, and that $w(i) = i$ for all $i \equiv 0 \pmod{N}$.  Given a mirrored $\mathbb{Z}$-permutation $w$, the ordered sequence $[w(1), w(2), \ldots, w(2n)]$ is called the \emph{window} of $w$.  The windows for the mirrored $\mathbb{Z}$-permutations corresponding to the Coxeter generators of $\widetilde{C}_n$ are shown below: 
\begin{eqnarray}
s_i &=& [1, 2, \ldots, i-1, i+1, i, i+2, \ldots, n] \text{ for } 1 \leq i \leq n-1,  \nonumber \\ 
s_0^C &=& [-1, 2, 3, \ldots, n],  \nonumber \\
s_n^C &=& [1, 2, \ldots, n-1, n+1],  \nonumber
\end{eqnarray}
The minimal length coset representatives in $\widetilde{C}_n/C_n$ satisfy $w(1) < w(2) < \cdots < w(n) < w(n+1)$.  Abacus diagrams combinatorialize the integers occurring in the base window of a mirrored $\mathbb{Z}$-permutation.

 

\begin{df} \rm (Definition 3.1 in \cite{HanusaJones})
An \emph{abacus diagram} is a diagram containing $2n$ columns labeled $1, 2, \ldots, 2n$, called \emph{runners}.  Runner $i$ contains entries labeled by the integers $mN + i$ for each level $m$ where $- \infty < m < \infty$. An example of an abacus diagram is shown in Figure~\ref{fig:exampleAbacus}.

An abacus diagram is drawn as follows: 
\begin{enumerate}[label = (\roman{*})]
\item Each runner is vertical, with $-\infty$ at the top and $\infty$ at the bottom.  The runners increase from runner 1 in the leftmost position to runner $2n$ in the rightmost position. 
\item Entries in the abacus diagram may be circled.  The circled entries are called \emph{beads}, and the non-circled entries are called \emph{gaps}.  The entries are linearly ordered by the labels $mN + i$, where $m \in \mathbb{Z}$ is the level and $1 \leq i \leq 2n$ is the runner number.  This linear ordering is called the \emph{reading order}.  
\item A bead $b$ is \emph{active} if there exist gaps that occur prior to $b$ in the reading order.  Otherwise, the bead $b$ is \emph{inactive}.  A runner is called \emph{flush} if no bead on the runner is preceded in reading order by a gap in the same runner.  An abacus is flush if every runner is flush.  
\item An abacus is \emph{balanced} if there is at least one bead on every runner, and the sum of labels of the lowest bead on runners $i$ and $N - i$ is $N$ for all $i = 1, 2, \ldots, 2n$ (equivalently, the sum of the highest levels that contain a bead for runners $i$ and $N-i$ is 0). 
\item An abacus is \emph{even} if there is an even number of gaps preceding $N$ in the reading order.
\end{enumerate}
A runner containing a bead at the highest level is called the \emph{largest runner}, and a runner whose lowest bead occurs at the lowest level among all runners is called the \emph{smallest runner}.  Note that largest and smallest runners may not be unique.   Denote by $\mathscr{A}_{2n}$ the set of all balanced flush abaci on $2n$ runners.
\end{df}

\begin{df}
Given a mirrored $\mathbb{Z}$-permutation $w$, define $\mathfrak{a}(w)$ to be the balanced flush abacus whose lowest bead in each runner is an element of $\{ w(1), w(2), \ldots, w(2n) \}$.
\end{df}

\begin{figure}[htp]
\begin{center}
\includegraphics[scale = 0.6]{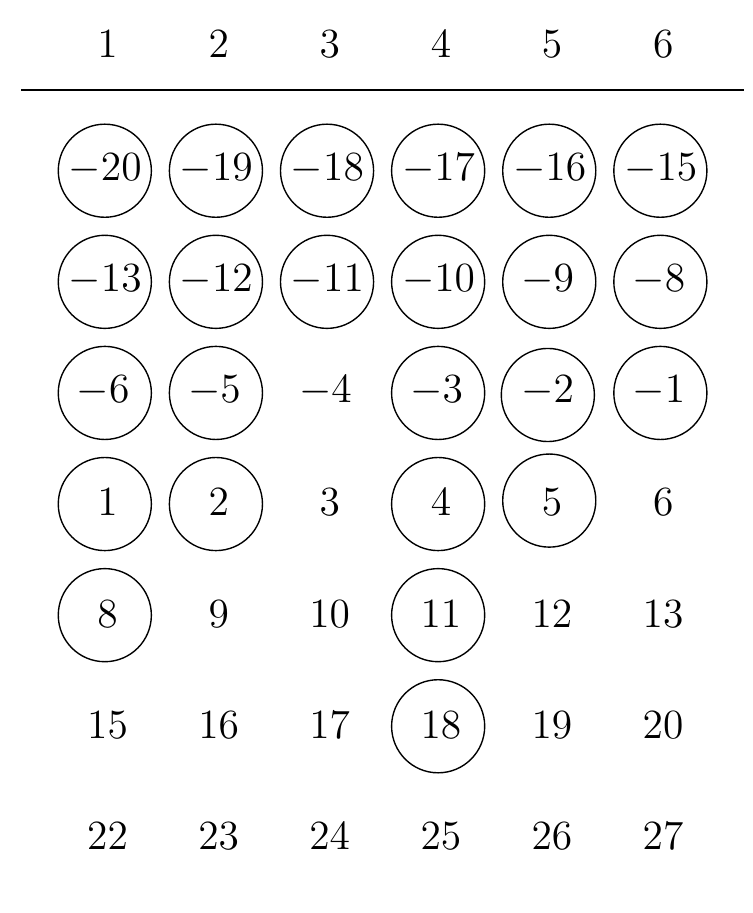}
\caption{The balanced flush abacus $\mathfrak{a}(w)$ corresponding to the mirrored $\mathbb{Z}$-permutation $w = [-11, -1, 2, 5, 8, 18]$. }
\label{fig:exampleAbacus}
\end{center}
\end{figure}

It has been shown (Lemma 3.6 in \cite{HanusaJones}) that the map $\mathfrak{a}: \widetilde{C}_n/C_n \longrightarrow \mathscr{A}_{2n}$ by $w \mapsto \mathfrak{a}(w)$ is a bijection.  From now on, we will assume all abacus diagrams to be balanced flush abacus diagrams unless otherwise noted. When we translate the action of the Coxeter generators on the mirrored $\mathbb{Z}$-permutations through the bijection $\mathfrak{a}$ to abacus diagrams, we get an action of the Coxeter generators on abacus diagrams.  They are described in Section 3.2 of \cite{HanusaJones}.  Since we are going to use these actions later in our proofs, we summarize them again here:  
\begin{enumerate}[label = (\roman{*})]
\item $s_i$ interchanges column $i$ with column $i+1$ and interchanges column $2n - i$ with column $2n - i + 1$, for $1 \leq i \leq n - 1$.  
\item $s_0^C$ interchanges column 1 and $2n$, and then shifts the lowest bead on column 1 down one level towards $\infty$, and shifts the lowest bead on column $2n$ up one level towards $- \infty$. 
\item $s_n^C$ interchanges column $n$ with column $n + 1$. 
\end{enumerate}

The following theorem shows that there is a bijection between coroot lattice points and abacus diagrams for $\widetilde{W}/W$. 
\begin{thm} [Theorem 4.1 in \cite{HanusaJones}] \label{T:corootcoords}
The coroot lattice point for an element $w \in \widetilde{W}/W$ is 
$$\sum_{i = 1}^n \operatorname{level}_i(\mathfrak{a}(w)) \varepsilon_i, $$
where $\operatorname{level}_i (\mathfrak{a}(w))$ is the level of the lowest bead in column $i$ of the abacus $\mathfrak{a}(w)$.  
\end{thm}

\subsection{Core Partitions}\label{sec:corepartitions}

Let $\lambda = (\lambda_1, \ldots, \lambda_r)$ be a partition.  The \emph{Young diagram} of $\lambda$ is a collection of left-justified boxes for which the number of boxes weakly decreases from $\lambda_1$ to $\lambda_r$ as one moves down the rows.  We will use $(i, j)$ to denote the box located at the $i^{\text{th}}$ row and the $j^{\text{th}}$ column of the Young diagram of $\lambda$.  The $(i, j)^{\text{th}}$ \emph{hook length} of $\lambda$, denoted by $h^{\lambda}_{(i, j)}$, is the number of boxes to the right and below the $(i, j)^{\text{th}}$ box of $\lambda$, including the $(i, j)^{\text{th}}$ box itself.  A Young diagram is \emph{symmetric} if it is symmetric across the line formed by the boxes along the diagonal $(i, i)$.  

\begin{df}
Given an integer $n \geq 2$, a partition $\lambda$ is an \emph{$n$-core} if for every box $(i, j)$ in the Young diagram of $\lambda$, we have $n \nmid h^{\lambda}_{(i, j)}$.  
\end{df}

Following the notation in \cite{BJV}, we denote the set of all $n$-cores by $\mathcal{C}_{n}$, the set of all $n$-cores with first part equal to $k$ by $\mathcal{C}_{n}^k$, and the set of all $n$-cores with first part $\leq k$ by $\mathcal{C}_{n}^{\leq k}$.  Similarly, we will denote the set of all symmetric $n$-cores by $\mathscr{S}_{n}$, the set of all symmetric $n$-cores with first part equal to $k$ by $\mathscr{S}_{n}^k$, and the set of all symmetric $n$-cores with first part $\leq k$ by $\mathscr{S}_{n}^{\leq k}$.  

There is a bijection between balanced flush abacus diagrams with $2n$ runners and $(2n)$-cores.  The bijection $F_{\mathscr{S}}: \mathscr{A}_{2n} \to \mathscr{S}_{2n}$ is defined as follows.

\begin{df} \rm (Definition 5.2 in \cite{HanusaJones})
Let $\mathfrak{a} \in \mathscr{A}_{2n}$ be a balanced flush abacus with $2n$ runners and $M$ active beads.  Define $F_{\mathscr{S}}(\mathfrak{a})$ to be the partition whose $i^{\text{th}}$ row contains the same number of boxes as gaps that appear before the $(M - i + 1)^{\text{th}}$ active bead in reading order.  It is shown in Section 5 of \cite{HanusaJones} that the image of $F_{\mathscr{S}}$ is indeed the set of symmetric $(2n)$-cores.  
\end{df}

A summary of the three combinatorial models for $\widetilde{C}_n/C_n$ discussed in this section appears in Table~\ref{tab:Cmodels}.

\begin{table}[h]
\caption{Models for $\widetilde{W}/W$ in Lie type $C$}
\label{tab:Cmodels}
\begin{tabular}{c l }
\hline\noalign{\smallskip}
Model for $\widetilde{C}_n/C_n$ & Conditions \\
\noalign{\smallskip}\hline\noalign{\smallskip}
Core Partitions & symmetric $(2n)$-cores \\
Abacus Diagrams & balanced flush abaci on $2n$ runners  \\  
Coroot Lattice Points & $(a_1, \ldots, a_n) \in \mathbb{Z}^n$  \\ 
\noalign{\smallskip}\hline
\end{tabular}
\end{table}

\subsection{Canonical Reduced Words}\label{S:canonicalwords}

Under the bijection $F_{\mathscr{S}}$ between abacus diagrams and cores, we get a natural action of $\widetilde{C}_n$ on the set of symmetric $(2n)$-cores.  To describe this action, we begin by labeling $\mathbb{N}^2$ so that $(i, j)$ corresponds to row $i$ and column $j$ in the Young diagram.  Define the \emph{residue} of a position in $\mathbb{N}^2$ to be 
$$\text{res}(i, j) = \left \{\begin{array}{l l } (j - i) ||(2n) & \text{ if } 0 \leq (j - i) || (2n) \leq n \\ 
2n - ((j-i)||(2n)) & \text{ if } n < (j -i ) || (2n) < 2n \end{array} \right. $$
where $p || q$ is the integer in $\{ 0, 1, \ldots, q -1 \}$ that is congruent to $p$ mod $q$.  When a Young diagram is labeled with its residues, we call the boxes in the Young diagram containing residue $i$ the \emph{$i$-boxes} of the diagram.  

We say that a box is \emph{addable} to a partition $\lambda$ if adding the box to the Young diagram of $\lambda$ results in a partition.  Similarly, a box is \emph{removable} if removing the box from $\lambda$ results in a partition.  
\begin{thm}[Theorem 5.8 in \cite{HanusaJones}]\label{C_nCoreAction}
Let $s_i \in \widetilde{C}_n$ be a generator of the Coxeter group $\widetilde{C}_n$.  If there exist addable $i$-boxes or removable $i$-boxes, then $s_i$ acts on $\lambda$ by adding all addable $i$-boxes to $\lambda$, or deleting all removable $i$-boxes from $\lambda$.  If there are no addable or removable $i$-boxes in $\lambda$, then $s_i$ acts as the identity on $\lambda$.  This provides a well-defined action of $\widetilde{C}_n$ on $\mathscr{S}_{2n}$.
\end{thm}

Using Theorem~\ref{C_nCoreAction}, we may repeatedly delete removable boxes from a (2n)-core $\lambda$ to obtain a canonical reduced word corresponding to $\lambda$, as the following example illustrates.  

\begin{exam} 
Taking a 4-core and repeatedly applying Theorem~\ref{C_nCoreAction} as shown in Figure~\ref{fig:exampleC_2ActOnCore}, we obtain the reduced word $s_0s_1s_0s_2s_1s_0$.
\begin{figure}[htp]
\begin{center}
\subfigure[$s_0\lambda$]{\includegraphics[scale = 0.5]{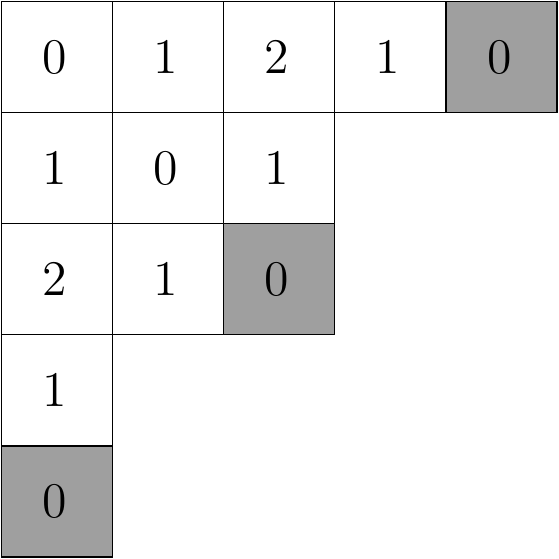}} \hspace{0.2in}
\subfigure[$s_1s_0\lambda$] {\includegraphics[scale = 0.5]{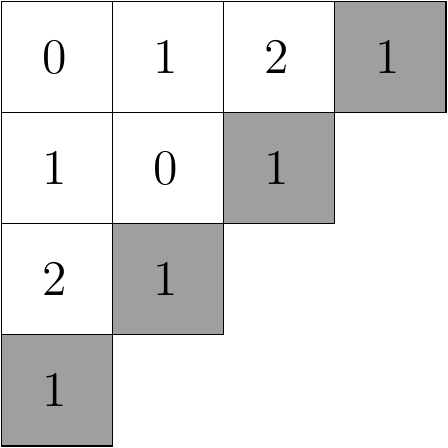}} \hspace{0.2in}
\subfigure[$s_0s_1s_0 \lambda$] {\includegraphics[scale = 0.5]{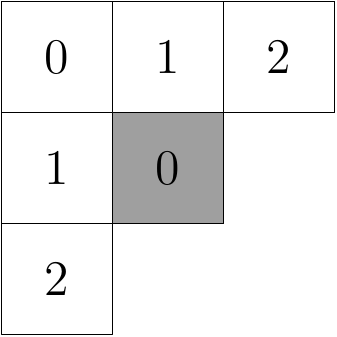}} \hspace{0.2in}
\subfigure[$s_2s_0s_1s_0 \lambda$] {\includegraphics[scale = 0.5]{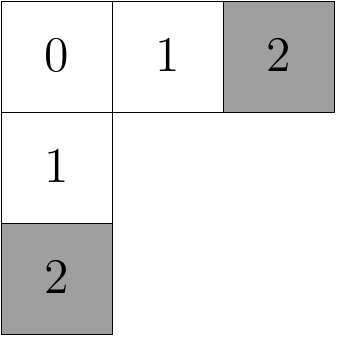}} \hspace{0.2in}
\subfigure[$s_1s_2s_0s_1s_0 \lambda$] {\includegraphics[scale = 0.5]{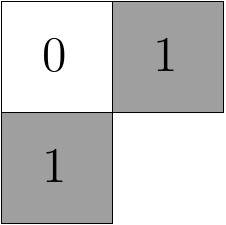}} \hspace{0.2in}
\subfigure[$s_0s_1s_2s_0s_1s_0 \lambda$] {\includegraphics[scale = 0.5]{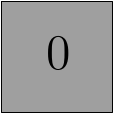}}
\caption{The action of $\widetilde{C}_2$ on a $4$-core and the canonical reduced word. }
\label{fig:exampleC_2ActOnCore}
\end{center}
\end{figure}
\end{exam}


\section{The Projection Map $\Phi_n$}\label{sectionDefining}

\subsection{Review of the Results for Type $A$} \label{sectionReview}

In this section, we review results obtained for the projection map $\Phi_n: \widetilde{S}_n/S_n \to \widetilde{S}_{n-1}/S_{n-1}$ in \cite{BJV}.  Core partitions, abacus diagrams, and root lattice points also index elements of $\widetilde{S}_n/S_n$ in Lie type $A$.  These models are the symmetric group analogues of the models described in Section~\ref{sectionCombinatorialModels} for type $\widetilde{C}_n$, and they are summarized in Table \ref{tab:Amodels}.

\begin{table}[h]
\caption{Models for Lie type $A$}
\label{tab:Amodels}
\begin{tabular}{c l }
\hline\noalign{\smallskip}
Model for $\widetilde{S}_n/S_n$ & Conditions \\
\noalign{\smallskip}\hline\noalign{\smallskip}
Core Partitions & $n$-cores \\
Abacus Diagrams & $n$ runners and sum of the highest levels that contain a bead equals 0 \\  
Root Lattice Points & $(a_1, \ldots, a_n) \in \mathbb{Z}^n$ such that $\sum_{i = 1}^n a_i = 0$ \\ 
\noalign{\smallskip}\hline
\end{tabular}
\end{table}

We first define a map $\Phi_n^k: \widetilde{S}_n/S_n \to \widetilde{S}_{n-1}/S_{n-1}$ on core partitions.  Given an $n$-core $\lambda$, consider its Young diagram.  To apply $\Phi_n^k$, first compute all of the hook lengths $h_{(i, 1)}$ of the left most squares of the $i^{\text{th}}$ row.  Then, delete all rows $i$ of $\lambda$ for which $h_{(i, 1)} \equiv h_{(1, 1)} \pmod{n}$.  Using abacus diagrams, Berg, Jones, and Vazirani show that when $\Phi_n^k$ is applied to a $n$-core $\lambda$ with first part equal to $k$, the resulting partition $\Phi_n^k(\lambda)$ is a $(n-1)$-core with first part at most $k$.  Furthermore, it is shown in \cite{BJV} that the map $\Phi_n^k: \mathcal{C}_{n}^k \to \mathcal{C}_{n-1}^{\leq k}$ is a bijection. 

Given an abacus $\mathfrak{a}$ corresponding to an $n$-core in $\mathcal{C}_{n}^k$, the map $\Phi_n^k$ acts by removing the entire runner with the largest bead.  To obtain an abacus corresponding to an $(n-1)$-core, simply place the remaining runners onto an abacus on $n-1$ runners, keeping the levels of the entries the same as they were prior to removing the largest runner.

When defined on root lattice points, the map $\Phi_n^k$ becomes more geometrically enlightening.  As a review, recall that the simple roots $\Delta$ of type $A_{n-1}$ are the $n-1$ vectors 
$$\alpha_1 = \varepsilon_1 - \varepsilon_2, \, \, \alpha_2 = \varepsilon_2 - \varepsilon_3, \, \, \ldots, \, \, \alpha_{n-1} = \varepsilon_{n-1} - \varepsilon_n.$$  
In type $A$, the $n$-cores correspond to (co)root lattice points $(a_1, \ldots, a_n) \in \Lambda_R^{\vee}$, where $a_i \in \mathbb{Z}$ and $\sum_{i = 1}^n a_i = 0$; refer to Remark~\ref{rem:coroot} for a disambiguation of the terminology.  Let $V =  \Lambda_R^{\vee}  \otimes_{\mathbb{Z}} \R \subsetneq \mathbb{R}^n$.  As summarized in the table above, elements of $V$ are $(a_1, \ldots, a_n) \in \mathbb{R}^n$ such that $\sum_{i = 1}^n a_i = 0$.  When the cores are identified with the root lattice points, the domain $\mathcal{C}_n^k$ of $\Phi_n^k$ lies inside a hyperplane in $V$.  More specifically, for $k \geq 0$, let $H_{n}^k$ denote the affine hyperplane 
\begin{equation}\label{eq:Ahyperplane}H_{n}^k = \left\{ v \in \mathbb{R}^n: (v, \varepsilon_{(k \modulo n)}) = \left\lceil \frac{k}{n} \right\rceil \right\} \cap V, \end{equation}
where $1 \leq (k \modulo n) \leq n$.   Corollary 3.2.15 in \cite{BJV} says that the $n$-cores $\lambda$ with first part $\lambda_1 = k$ all lie inside $H_{n}^k \cap \Lambda_R^{\vee}$.  Further, Theorem 4.1.1 in \cite{BJV} says that the map $\Phi_{n}^k$, when restricted to the domain $\mathcal{C}_{n}^k$ inside $H_{n}^k \cap {\Lambda_{R}}^{\vee}$, is a projection onto the hyperplane $H_{n}^k$.


\subsection{The Map $\Phi_n$ on Symmetric Cores}\label{MapOnCores}

The map ${\Phi_n}$ acting on the set $\mathscr{S}_{2n}$ of symmetric $(2n)$-cores can be defined in the following way. First, label the boxes of a $(2n)$-core by the the elements of $\Z/2n\Z$ repeating along diagonals by assigning box $(i, j)$ the label $j - i \pmod{2n}$.  Then, delete all the rows and columns that end with the same element of $\Z/2n\Z$ as the first row (equivalently column).  An example of this process is shown below: 
 
\begin{figure}[htp]
\begin{center}
\subfigure{\includegraphics[scale = 0.5]{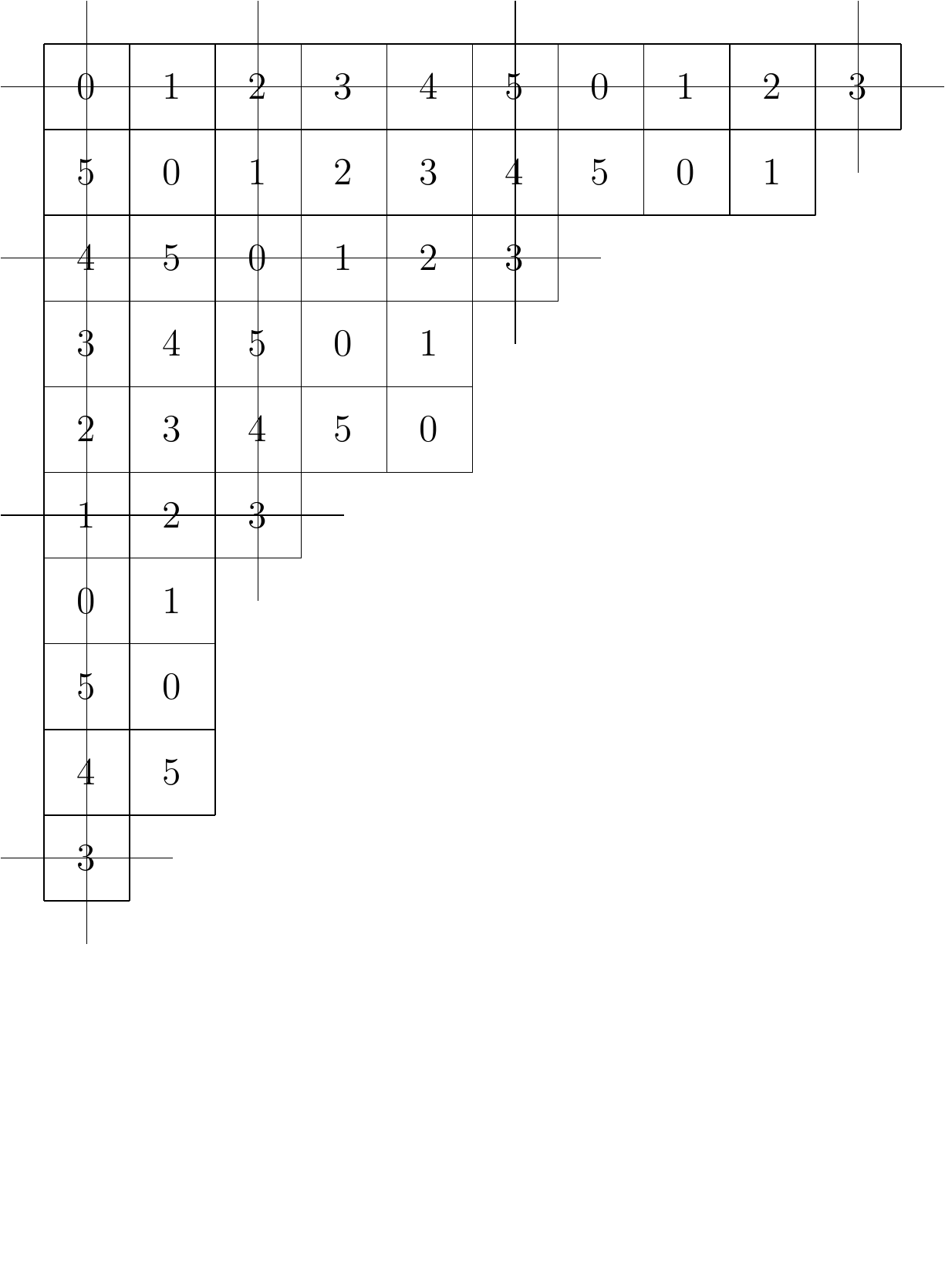}} \hspace{0.3in}
\subfigure{\includegraphics[scale = 0.5]{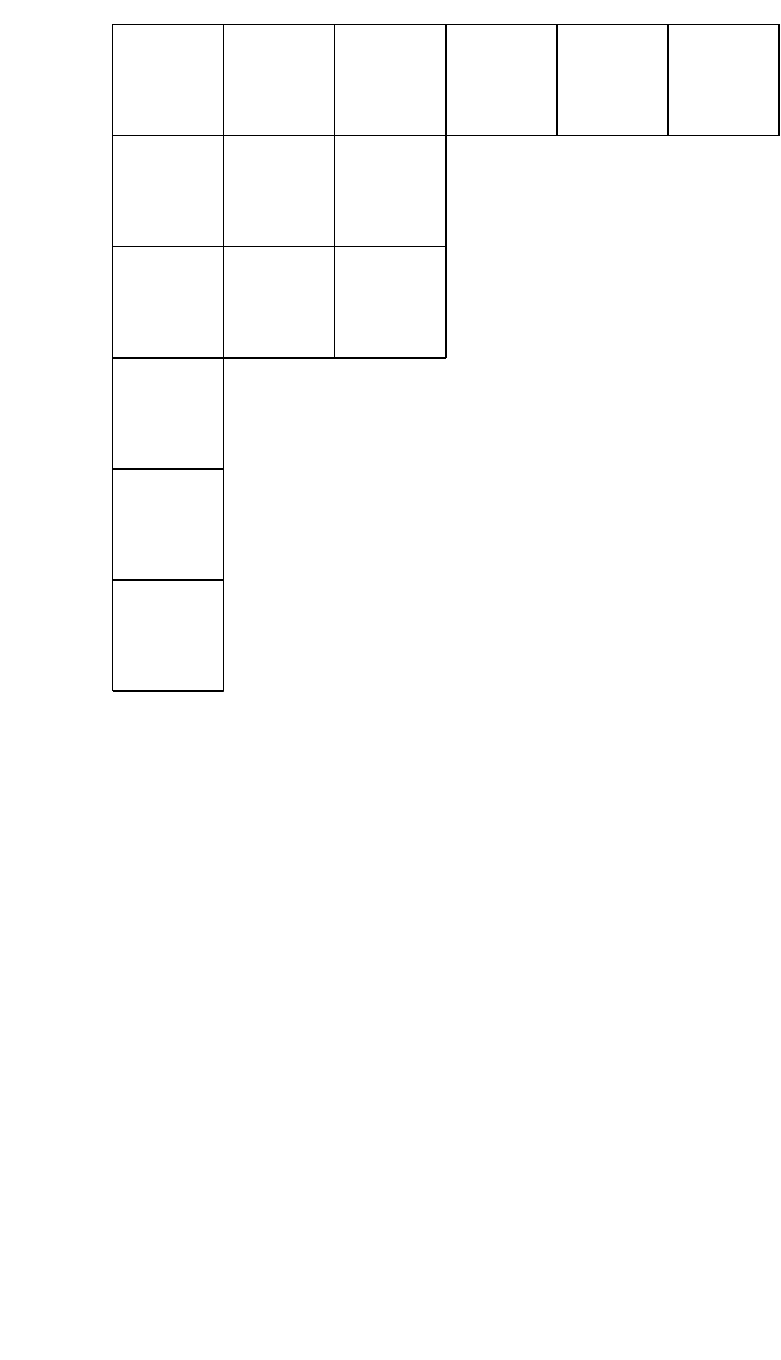}} 
\vspace{-1in}
\caption{The map $\Phi_3$ acting on a symmetric 6-core}
\label{figPhiOnCores}
\end{center}
\end{figure}

\subsection{The Map $\Phi_n$ on Abacus Diagrams}

Given an abacus $\mathfrak{a}$, write $\mathfrak{a} = (a_1, \ldots, a_n, -a_n, \ldots, -a_1)$, where we denote by $a_i$ the level of the lowest bead in column $i$ so that the $i^{\text{th}}$ coordinate of $\mathfrak{a}$ equals $\operatorname{level}_i(\mathfrak{a})\varepsilon_i$ for $1 \leq i \leq n$. Note that these are precisely the coordinates of the coroot lattice point corresponding to $\mathfrak{a}$ by Theorem \ref{T:corootcoords}.  

Define $\Phi'_n(\mathfrak{a})$ to be the abacus obtained via the following procedure: first, locate the right-most runner with the largest coordinate  in abacus $\mathfrak{a}$.  Then, delete this runner \textit{and} its symmetric runner.  The resulting abacus $\Phi'_n(\mathfrak{a})$ is a balanced abacus with $2n-2$ runners, corresponding to an element of $\widetilde{C}_{n-1}/C_{n-1}$.  An example of this procedure is shown in Figure \ref{figPhiOnAbacus} (the abacus corresponds to the core partition in Figure~\ref{figPhiOnCores}).
\begin{figure}[htp]\label{figPhiOnAbacus}
\begin{center}
\subfigure[$\mathfrak{a} = (1, 2, -2, 2, -2, -1)$]{\includegraphics[scale = 0.65]{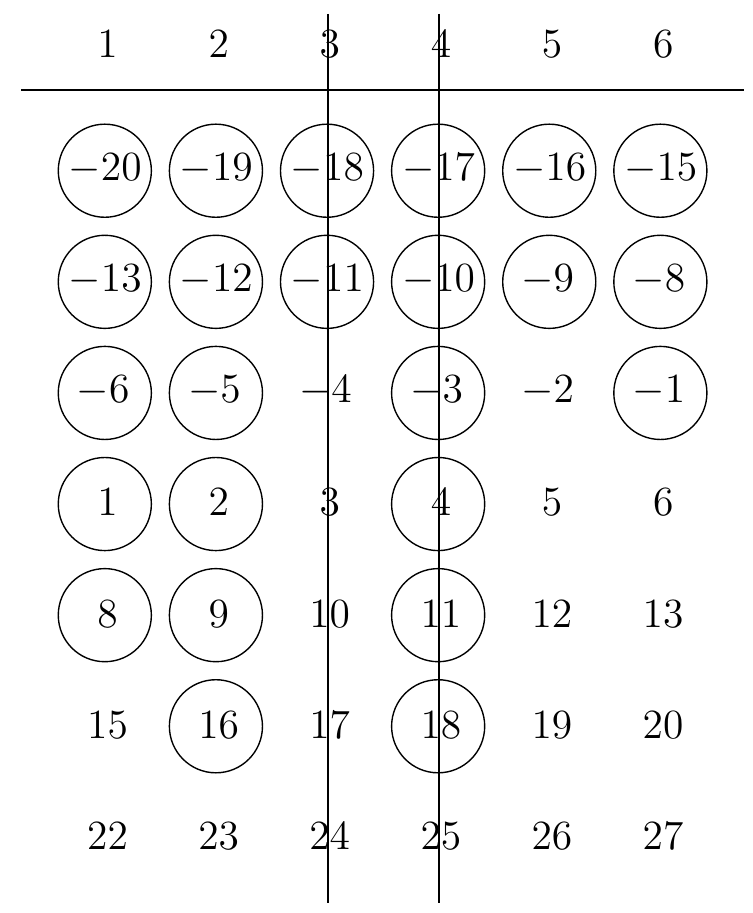}} \hspace{0.3in}
\subfigure[$\Phi'_3(\mathfrak{a}) = (1,2,-2,-1)$]{\includegraphics[scale = 0.65]{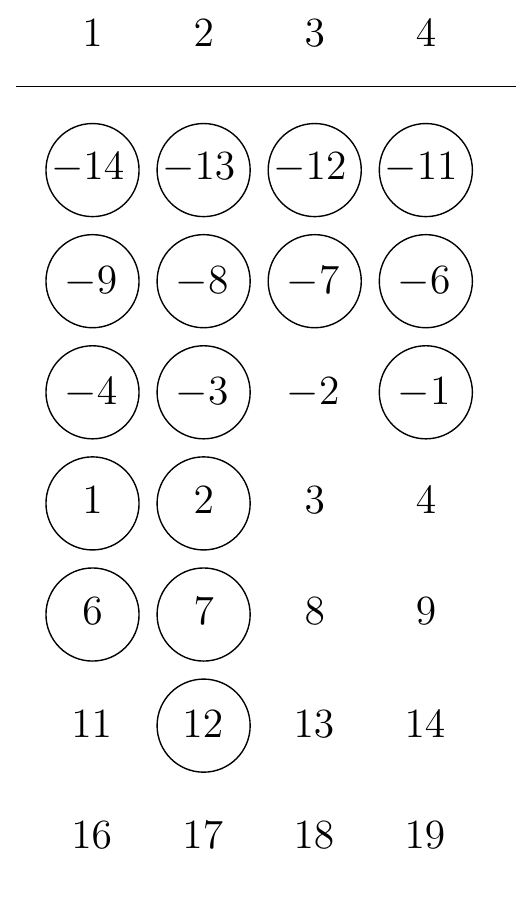}} 
\caption{Action of $\Phi'_3$ on the abacus $\mathfrak{a} = (1, 2, -2, 2, -2, -1)$.}
\label{figPhiOnAbacus}
\end{center}
\end{figure}

As reviewed in Section \ref{sec:corepartitions}, balanced abacus diagrams for the quotient $\widetilde{C}_n/C_n$ are in bijective correspondence with symmetric $(2n)$-cores.  Let this bijection be denoted by $F_{\mathscr{A}}$.  Using the bijection $F_{\mathscr{A}}$ and the map $\Phi_n$ on $\mathscr{S}_{2n}$, we get an induced map $\Phi_n: \mathscr{A}_{2n} \to \mathscr{A}_{2n-2}$ on abacus diagrams such that the following diagram is commutative: 
\begin{eqnarray}\label{coreToAbacus}
\xymatrix{
\mathscr{S}_{2n} \ar[r]^{F_{\mathscr{A}}}  \ar[d]^{\Phi_n} & \mathscr{A}_{2n} \ar[d]^{\Phi_n} \\ 
\mathscr{S}_{(2n-2)} \ar[r]^{F_{\mathscr{A}}} & \mathscr{A}_{(2n-2)}  
}
\end{eqnarray}

\begin{prop}\label{delRowCol}
The induced map $\Phi_n$ on abacus diagrams is the map ${\Phi'_n}$.  
\end{prop}

To prove this proposition, we will need the following lemma: 

\begin{lem}\label{delRowColLem}
In an abacus diagram $\mathfrak{a}$ corresponding to the core partition $F_{\mathscr{S}}(\mathfrak{a})$, there is a left-most runner whose largest bead is located at the smallest level.  Deleting this runner results in an abacus $\mathfrak{a}'$ whose corresponding core partition is obtained from $F_{\mathscr{S}}(\mathfrak{a})$ by deleting all columns that end with the same element of $\Z/2n\Z$ as the first column. 
\end{lem}
\begin{proof}
By construction of the core partition from the abacus diagram, the columns correspond to the gaps that are smaller than the largest active bead.  For example, the smallest gap, which is smaller than every active bead, corresponds to the first column.  It can be easily checked (similar to the case of deleting the right-most runner having a bead at the highest level) that two columns end with the same element of $\Z/2n\Z$ if and only if their gaps are in the same runner.  The left-most runner whose largest bead is the smallest is the runner containing the smallest gap.  By deleting this runner, we are deleting all the columns that end with the same element of $\Z/2n\Z$ as the first column, as claimed.
\end{proof}

We are now ready to prove Proposition~\ref{delRowCol}. 

\begin{proof}[Proof of Proposition \ref{delRowCol}]
Notice that the symmetric runner of the right-most largest runner is the left-most smallest runner.  To apply $\Phi'_n$, we will proceed in two steps.  Deleting the right-most largest runner corresponds to deleting all rows that end in the same element of $\Z/2n\Z$ as the first row.  Deleting the symmetric runner, by Lemma~\ref{delRowColLem}, corresponds to deleting all columns that end with the same element of $\Z/2n\Z$  as the first column.  As in the proof of Lemma~\ref{delRowColLem}, the columns correspond to gaps in the symmetric runner that are smaller than the largest active bead.  It follows that the resulting $(2n-2)$-core partition we get from applying $\Phi'_n$ is what we would obtain if we apply the map $\Phi_n$.
\end{proof}

\subsection{The Domain and Codomain of the Map $\Phi_n$}\label{DomainCodomain}

In order to fully understand the map $\Phi_n: \widetilde{C}_n /C_n \to \widetilde{C}_{n-1}/C_{n-1}$, it is necessary to use the theory of alcoves.  We have seen  that for type $A$, the map $\Phi_n^k$ is a projection when restricted to the root lattice points corresponding to $n$-cores in $\mathcal{C}_n^k$.  In general, given a parabolic quotient $\widetilde{W}/W$, we wish to partition its elements into hyperplane domains.  The cores corresponding to each hyperplane domain should satisfy common combinatorial properties.  By identifying each hyperplane with $\widetilde{C}_{n-1}/C_{n-1}$, the map $\Phi_n: \widetilde{C}_n/C_n \to \widetilde{C}_{n-1}/C_{n-1}$ is subsequently defined by projecting each domain onto their hyperplanes.   In this subsection, we wish to find a partition of the domain of $\Phi_n$ so that when restricted to these parts, the map $\Phi_n$ is bijective. 

\begin{lem}\label{lemLambda1}
Let $\mathfrak{a} \in \mathscr{A}_{2n}$ be a balanced flush abacus with $2n$ runners.  If the largest bead of $\mathfrak{a}$ is located at level $\ell$ of runner $i$, where $1 \leq i \leq 2n$, then the first part of the symmetric $(2n)$-core $F_{\mathscr{S}}(\mathfrak{a})$ is $\lambda_1 = 2n(\ell - 1) + i$. 
\end{lem}
\begin{proof}
The first part of the $(2n)$-core $F_{\mathscr{S}}(\mathfrak{a})$, say $\lambda_1$, is equal to the number of gaps that are smaller than the largest bead.  Let the largest bead of each runner be located at levels $(r_1, r_2, \ldots, r_{2n}) = (a_1, \ldots, a_n, -a_n, \ldots, -a_1)$, respectively.  Since the largest bead is located at level $\ell$ of runner $i$ (i.e. $r_i = \ell$), the number of gaps that occur before this bead is
\begin{eqnarray}
\sum_{j = 1}^{i} (\ell - r_j) + \sum_{j = i+1}^{2n} (\ell - r_j - 1) &=& 2n \ell - \sum_{j = 1}^{2n} r_j - (2n - i)  \nonumber \\ 
&=& 2n \ell - (2n - i ) = 2n(\ell - 1) + i,  \nonumber
\end{eqnarray}
as desired. 
\end{proof}

Let $\Phi_n^k$ denote the map $\Phi_n$ when the domain is restricted to $\mathscr{S}_{2n}^{k}$.

\begin{cor}\label{corLambda1}
Under the bijection between symmetric $(2n)$-cores and balanced flush abaci, $F_{\mathscr{A}}(\mathscr{S}_{2n}^k)$ is the set of abaci where the largest runner is located at level $\ell = \left\lceil \frac{k}{2n}\right\rceil$ of runner $i := k \pmod{2n}$.
\end{cor}

\begin{proof}
We may write $k$ uniquely as $2n(\ell - 1) + i$, where $\ell \geq 1$ and $1 \leq i \leq 2n$.  By Lemma~\ref{lemLambda1}, elements of $\mathscr{S}_{2n}^k$ correspond to balanced flush abaci where the largest active beads are located at level $\ell$ of the $i^{\text{th}}$ runner. 
\end{proof}

\begin{prop}\label{imageContainedCore}
The image of $\Phi_n^k$ is a subset of the set $\mathscr{S}_{2n-2}^{\leq \left(k - \left\lceil \frac{k}{n} \right\rceil \right)}$ of $(2n-2)$-cores with first part at most $k - \left\lceil \frac{k}{n} \right\rceil$.
\end{prop}

\begin{proof}
To show that the image of the map $\Phi_n^k$ is contained in $\mathscr{S}_{2n-2}^{\leq k - \left\lceil \frac{k}{n} \right\rceil}$, notice that the number of gaps in runner $2n+1 - i$ which are smaller than the largest active bead (which is located in the $i^{\text{th}}$ runner) is equal to $\lceil \frac{k}{n} \rceil$.  Indeed, it is equal to $2\ell - 1$ if $i \equiv 1, \ldots, n \pmod{2n}$, and $2 \ell$ if $i \equiv n+1, \ldots, 2n \pmod{2n}$.  When we apply the map $\Phi_n^k$ to this abacus, the number of gaps that are less than the first active bead decreases by at least $\left\lceil \frac{k}{n} \right\rceil$.  Hence, the image of $\Phi_n^k$ is a subset of $\mathscr{S}_{2n-2}^{\leq k - \left\lceil\frac{k}{n} \right\rceil }$.  
\end{proof}

\subsection{The Map $\Phi_n$ on Coroot Lattice Points}\label{MapCoroot}

Recall that there is a bijective correspondence between elements in $\widetilde{C}_n/C_n$ and coroot lattice points of the form $(a_1, \ldots, a_n) \in \mathbb{Z}^n$.  By Theorem \ref{T:corootcoords}, the coroot lattice point $(a_1, \ldots, a_n)$ corresponds to the abacus on $2n$ runners with the largest active bead for each 
runner located at levels $(a_1, \ldots, a_n, -a_n, \ldots, -a_1)$.

Fix an integer $k > 0$.  Define
\begin{equation}\label{E:ell} \ell_1 := k \pmod{n} \quad \text{and} \quad \ell_2 := k \pmod{2n},
\end{equation}
 where $1 \leq \ell_1 \leq n$ and $1 \leq \ell_2 \leq 2n$.   Let ${H_{n}^k}$ denote the affine hyperplane
\begin{equation} \label{E:Hkn}
H_n^k =  
\begin{cases}
\left\{ v \in \mathbb{R}^n: \langle v, \varepsilon_{\ell_1} \rangle = \left\lceil \frac{k}{2n} \right\rceil \right\} & \quad \text{if } \ell_2 \in \{ 1, \ldots, n\} \\ 
\left\{v \in \mathbb{R}^n: \langle v, \varepsilon_{n - \ell_1 +1} \rangle = - \left\lceil \frac{k}{2n} \right\rceil  \right\} & \quad \text{if } \ell_2 \in \{ n+1, \ldots, 2n\} .
\end{cases} 
\end{equation}
Note that all of the points on $H_n^k$ have the same fixed $\ell_1^{\text{th}}$ or $(n-\ell_1+1)^{\text{th}}$ coordinate. In particular, for the first case, the $\ell_1^{\text{th}}$ coordinate is $ \left\lceil\frac{k}{2n}\right\rceil$, and for the second case, it equals $ -\left\lceil\frac{k}{2n}\right\rceil$.  The affine hyperplanes $H^k_n$ for $\widetilde{C}_2/C_2$ are shown in Figure \ref{figHyperplanes}.

\begin{figure}[hpt]
\begin{center}
\includegraphics[scale = .8]{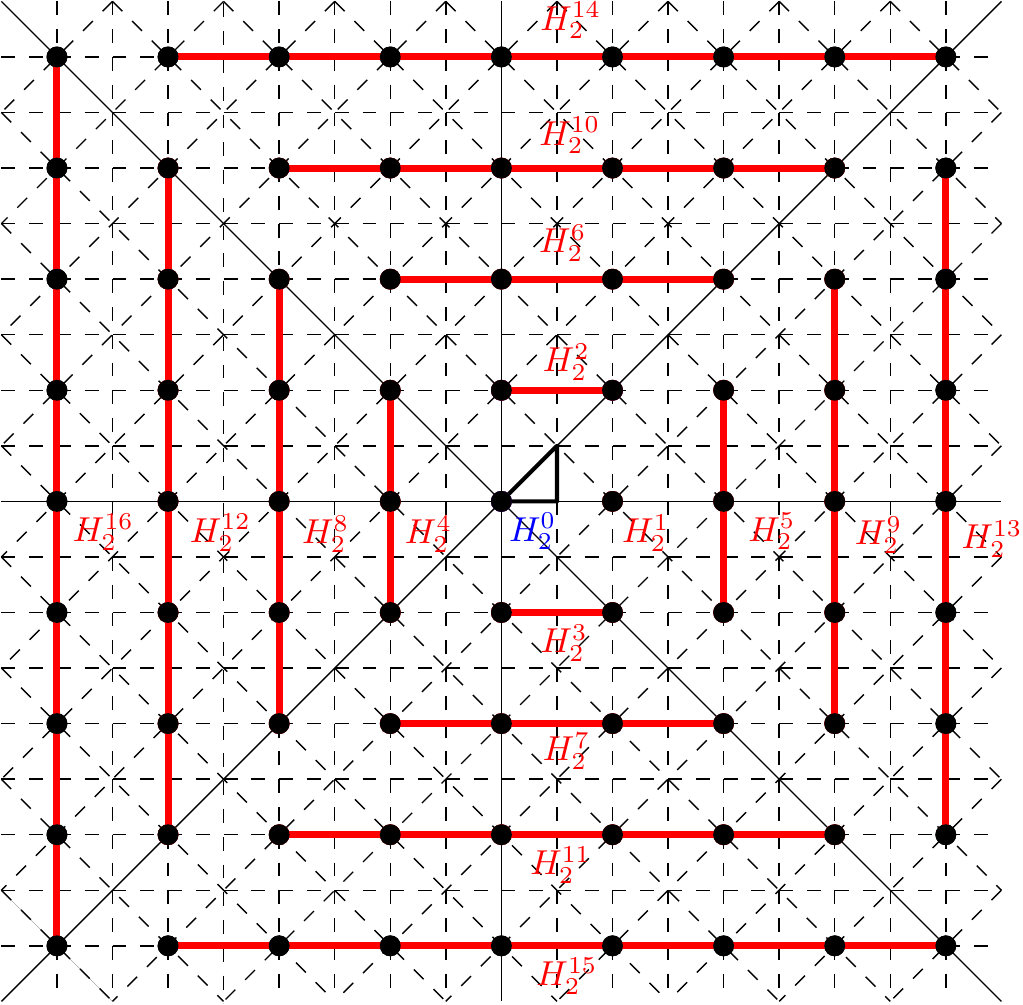}
\caption{The hyperplanes $H^k_2$.}
\label{figHyperplanes}
\end{center}
\end{figure}

\begin{prop} [Elements of $\mathscr{S}_{2n}^k$ lie on a hyperplane]\label{domainHyperplane}
Under the correspondence between symmetric $(2n)$-cores and coroot lattice points, we have the following: 
\begin{enumerate}
\item[(i)] If $\ell_2 \in \{ 1, \ldots, n\}$, then the symmetric $(2n)$-cores $\lambda$ with $\lambda_1 = k$ correspond to the lattice points $(a_1, \ldots, a_n) \in H_n^k \cap \mathbb{Z}^n$ subject to the conditions 
\begin{eqnarray}
- \left\lceil \frac{k}{2n} \right\rceil < a_i \leq \left\lceil \frac{k}{2n} \right\rceil, &&  i \in [1, \ell_1 - 1],  \nonumber \\ 
- \left\lceil \frac{k}{2n} \right\rceil < a_i < \left\lceil \frac{k}{2n} \right\rceil, && i \in [\ell_1 + 1, n].  \nonumber 
\end{eqnarray}

\item[(ii)] If $\ell_2 \in \{ n+1, \ldots, 2n \}$, then the symmetric $(2n)$-cores $\lambda$ with $\lambda_1 = k$ correspond to the lattice points $(a_1, \ldots, a_k) \in H_n^k \cap \mathbb{Z}^n$ subject to the conditions 
\begin{eqnarray}
-\left\lceil \frac{k}{2n} \right\rceil < a_i \leq \left\lceil \frac{k}{2n} \right\rceil, && i \in [1, n - \ell_1],  \nonumber \\ 
- \left\lceil \frac{k}{2n} \right\rceil \leq a_i \leq \left\lceil \frac{k}{2n} \right\rceil, && i \in [n - \ell_1 + 2, n]. \nonumber
\end{eqnarray}
\end{enumerate}

Under the correspondence between $(2n-2)$-cores and  lattice points in $\mathbb{Z}^{n-1}$, we have the following: 
\begin{enumerate}
\item[(iii)] If $\ell_2 \in \{ 1, \ldots, n\}$, then elements of $\mathscr{S}_{2n-2}^{\leq \left(k - \left\lceil \frac{k}{n} \right\rceil \right)}$ correspond to  lattice points $(a_1, \ldots, a_{n-1}) \in \mathbb{Z}^{n-1}$ subject to the conditions
\begin{eqnarray}
- \left\lceil \frac{k}{2n} \right\rceil < a_i \leq \left\lceil \frac{k}{2n} \right\rceil, &&  i \in [1, \ell_1 - 1],  \nonumber \\ 
- \left\lceil \frac{k}{2n} \right\rceil < a_i < \left\lceil \frac{k}{2n} \right\rceil, && i \in [\ell_1, n-1].  \nonumber
\end{eqnarray}

\item[(iv)] If $\ell_2 \in \{ n+1, \ldots, 2n\}$, then elements of $\mathscr{S}_{2n-2}^{\leq \left(k - \left\lceil \frac{k}{n} \right\rceil \right)}$ correspond to  lattice points $(a_1, \ldots, a_{n-1}) \linebreak \in \mathbb{Z}^{n-1}$ subject to the conditions
\begin{eqnarray}
-\left\lceil \frac{k}{2n} \right\rceil < a_i \leq \left\lceil \frac{k}{2n} \right\rceil, && i \in [1, n - \ell_1],  \nonumber \\ 
- \left\lceil \frac{k}{2n} \right\rceil \leq a_i \leq \left\lceil \frac{k}{2n} \right\rceil, && i \in [n - \ell_1 + 1, n]. \nonumber
\end{eqnarray}
\end{enumerate}
\end{prop}

\begin{proof}
The claims in cases ($i$) and ($ii$) follow from Corollary~\ref{corLambda1}.

We will prove the claim for case ($iii$), in which $1 \leq \ell_2 \leq n$.  The coroot lattice points satisfying condition ($iii$) above correspond to all balanced flush abaci whose highest active beads are no higher than the bead located at level $\left\lceil \frac{k}{2n} \right \rceil$ of runner $\ell_1 - 1$.  

Since $1 \leq \ell_2 \leq n$, we may write $k = 2n \ell + \ell_1$.  It follows from Lemma~\ref{lemLambda1} that the first part of the core partition is at most 
\begin{eqnarray}
(2n-2) \left(\left\lceil \frac{k}{2n}-1\right \rceil \right) + \ell_1 -1 &=& (2n-2)\ell + \ell_1 - 1 \nonumber \\ 
&=& (2n \ell + \ell_1) - (2 \ell + 1) \nonumber \\ 
&=& k - \left\lceil \frac{k}{n} \right\rceil,  \nonumber
\end{eqnarray}
as desired.  Case ($iv$) is proved analogously. 
\end{proof}

Denote the set of coroot lattice points corresponding to the core partitions $\mathscr{S}_{2n}^k$ by ${\mathscr{R}}_{2n}^k$, and the set of coroot lattice points corresponding to the core partitions $\mathscr{S}_{2n-2}^{\leq \left(k - \left\lceil \frac{k}{n} \right\rceil \right)}$ by $\mathscr{R}_{2n-2}^{\leq \left(k - \left\lceil \frac{k}{n} \right\rceil \right)}$.  We are now ready to prove our first main theorem, which appears as Theorem A in the introduction. 

\begin{thm}\label{MainTheorem1}
The map 
$$\Phi_n^k: \mathscr{S}_{2n}^k \to \mathscr{S}_{2n-2}^{\leq \left(k - \left\lceil \frac{k}{n} \right\rceil \right)}$$ 
is a bijection.
\end{thm}
\begin{proof}
Using the commutative diagram~\eqref{coreToAbacus} and the bijection $F_{\mathscr{R}}$ between abacus diagrams and coroot lattice points, the following diagram is commutative: 
\begin{eqnarray}\label{CoresToAbaci}
\xymatrix{
\mathscr{S}_{2n}^k \ar[r]^{F_{\mathscr{R}} \circ F_{\mathscr{A}} } \ar[d]^{\Phi_n^k} &{\mathscr{R}}_{2n}^k \ar[d]^{{\Phi}_n^k} \\ 
\mathscr{S}_{2n-2}^{\leq \left(k - \left\lceil \frac{k}{n} \right \rceil \right)} \ar[r]^{F_{\mathscr{R}} \circ F_{\mathscr{A}}} & {\mathscr{R}}_{2n-2}^{\leq \left(k - \left\lceil \frac{k}{n} \right\rceil \right)} 
}
\end{eqnarray}
In the commutative diagram above, $\Phi_n^k$ acts on the coroot lattice points in ${\mathscr{R}}_{2n}^k$ by deleting the $\ell_1^{\text{th}}$ coordinate if $1 \leq \ell_2 \leq n$, or deleting the $(n - \ell_1 + 1)^{\text{th}}$ coordinate if $n+1 \leq \ell_2 \leq n$. This is clearly an injection because in each case, the $\ell_1^{\text{th}}$ and $(n-\ell_1+1)^{\text{th}}$ coordinates are both fixed and redundant.

The map $\Phi_n^k$ is a bijection between ${\mathscr{R}}_{2n}^k$ and $\mathscr{R}_{2n-2}^{\leq \left(k - \left\lceil \frac{k}{n} \right\rceil \right)}$ because the inverse is defined by inserting $\left \lceil \frac{k}{2n} \right \rceil$ at position $\ell_1$ if $1 \leq \ell_2 \leq n$, or $- \left \lceil \frac{k}{2n} \right \rceil$ at position $n - \ell_1 + 1$ if $n+1 \leq \ell_2 \leq 2n$.  Since the horizontal arrows in the commutative diagram are also bijections, the map $\Phi_n^k: \mathscr{S}_{2n}^k \to \mathscr{S}_{2n-2}^{\leq \left(k - \left\lceil \frac{k}{n} \right\rceil \right)}$ is a bijection, as desired.
\end{proof}


\section{Geometric Interpretation of the map $\Phi_n^k$} \label{sectionGeometric}

In this section we describe a geometric interpretation of $\Phi_n^k$ on $\widetilde{C}_n/C_n$ using the alcove model. Recall the correspondence between $\C$ and alcoves in $\mathbb{R}^n$ discussed in Section \ref{sectionWeylGroups} given by sending a word $w\in\C$ to the alcove $w(\mathcal{A}_\circ)$.  

For the duration of this paper we fix two non-negative integers $n$ and $k$.  Moreover, whenever we are dealing with an object which is naturally indexed by $n$ and $k$, even if no label is explicitly mentioned, we will assume this fixed labeling.

Recall from \eqref{E:ell} that we have defined $\ell_2 := k \pmod{2n},$ where $1 \leq \ell_2 \leq 2n$, and that the definition of $H^k_n$ depends on the value of $\ell_2$. Throughout this section we shall only consider the case when $\ell_2\in\{1,\ldots,n\}$ since the proof for when $\ell_2\in\{n+1,\ldots,2n\}$ is completely analogous. Moreover, for ease of noations, $\ell_1 = k\pmod{n}$ will be denoted by $\ell$ for the rest of this section.

We begin by noting that we can naturally identify $H^k_n$ with $\mathbb{R}^{n-1}$ via the $\{\varepsilon_i:i\ne \ell\}$ basis where $\{\varepsilon_1,\ldots,\varepsilon_n\}$ is the standard orthonormal basis for $\mathbb{R}^n$. We define the map $\pi:\mathbb{R}^n\to H_n^k$ to be the projection of $\mathbb{R}^n$ onto $H_n^k$. Analytically, this is given by
\begin{equation}\label{defpi}
\pi(v)=\sum_{j\ne \ell}\langle v,\varepsilon_j\rangle \varepsilon_j+ \left\lceil\frac{k}{2n}\right\rceil\varepsilon_{\ell}.
\end{equation}

Suppose that
$\mathcal{A}_\circ=\mathcal{A}^1\to\cdots\to\mathcal{A}^r=\mathcal{A}_w$
is a minimal length alcove walk from $\mathcal{A}_\circ$ to $\mathcal{A}_w$. Now we aim to show that 
$\pi(\mathcal{A}^1)\to\cdots\to\pi(\mathcal{A}^r)$
is an alcove walk for $\mathcal{A}_{\Phi_n^k(w)}$, and that if one removes all repeated instances of alcoves in the projected walk, then the resulting walk is minimal. We begin by showing that  $\pi(\mathcal{A}^1)\to\cdots\to\pi(\mathcal{A}^r)$
is indeed an alcove walk for $\Phi_n^k(w)$. It is sufficient to show that the image of an alcove $\mathcal{A}\subseteq\mathbb{R}^n$ under $\pi$ is an alcove when we identify $H_n^k$ with $\mathbb{R}^{n-1}$.

\begin{lem}\label{alcove}
The image of an alcove under $\pi$ is an alcove via the identification of $H_n^k$ with $\mathbb{R}^{n-1}$.  
\end{lem}
\begin{proof}
Fix an alcove $\mathcal{A}$.  For every positive root $\beta$, there exists a unique integer $k_{\beta}$ such that $\lambda \in \mathbb{R}^n$ lies in the interior of the alcove $\mathcal{A}$ if and only if for all $\beta$, the point $\lambda$ satisfies the inequalities $k_{\beta} < (\lambda, \beta) < k_{\beta}+1$; \textit{i.e.} $\lambda$ lies between the hyperplanes $H_{\beta, k_{\beta}}$ and $H_{\beta, k_{\beta} + 1}$.  

Let $\alpha_i$ denote the ordered basis of simple roots in type $\widetilde{C}_n$.  The fundamental alcove is the region bounded by the hyperplanes $H_{\alpha_i, 0}$ for $1 \leq i \leq n$, and $H_{\widetilde{\alpha}, 1}$.  In other words, the fundamental alcove consists of all points $\lambda \in \mathbb{R}^n$ such that $(\lambda, \alpha_i) > 0$ for $1 \leq i \leq n$ and $(\lambda, \widetilde{\alpha}) < 1$.  This is precisely all points $(a_1, \ldots, a_n)\in \R^n$ satisfying the inequality 
$$\frac{1}{2} > a_1 > a_2 > \cdots > a_n > 0.  $$

Now, the alcove $\mathcal{A}$ is a translation by integer coordinates of an alcove based at the origin in the fundamental region, as defined in Section \ref{sectionWeylGroups}.  To show that the image of $\mathcal{A}$ is an alcove, it suffices to show that the image of any alcove in the fundamental region is an alcove.  All the alcoves in the fundamental region are obtained from the fundamental alcove via a sequence of reflections across the hyperplanes $H_{\beta, 0}$, where $\beta \in \Phi^+$ is a positive root. 

Recalling that the positive roots in type $\widetilde{C}_n$ are $2 \varepsilon_i$ for $1 \leq i \leq n$, and $\varepsilon_i \pm \varepsilon_j$ for $1 \leq i < j \leq n$, there are thus three types of reflections to consider: 

\begin{enumerate}[label=(\roman{*})]
\item Reflecting across the hyperplane $H_{\varepsilon_i - \varepsilon_j, 0}$.  In this case, we switch the $i^{\text{th}}$ coordinate and the $j^{\text{th}}$ coordinate; \textit{i.e.}, 
$$(a_1, \ldots, a_i, \ldots, a_j, \ldots, a_n) \mapsto (a_1, \ldots, a_j, \ldots, a_i, \ldots, a_n).$$ 
\item Reflecting across the hyperplane $H_{\varepsilon_i + \varepsilon_j, 0}$.  In this case, we first switch the $i^{\text{th}}$ coordinate and the $j^{\text{th}}$ coordinate, then also change their signs; \textit{i.e.},
$$(a_1, \ldots, a_i, \ldots, a_j, \ldots, a_n) \mapsto (a_1, \ldots, -a_j, \ldots, -a_i, \ldots, a_n).$$
\item Reflecting across the hyperplane $H_{2\varepsilon_i, 0}$.  In this case, only the sign of the $i^{\text{th}}$ coordinate changes; \textit{i.e.}, 
$$(a_1, \ldots, a_i, \ldots, a_n) \mapsto (a_1, \ldots, -a_i, \ldots, a_n).$$
\end{enumerate}

It is easy to see from the three cases above that an alcove in the fundamental region consists of all points $(k_1a_{\sigma(1)}, \ldots, k_na_{\sigma(n)}) \in \mathbb{R}^n$ such that 
$\frac{1}{2} > a_1 > a_2 > \cdots > a_n > 0, $
where $k_i \in \{-1, 1\}$ for $1 \leq i \leq n$, and $\sigma$ is a permutation of $\{ 1, 2, \ldots, n \}$.  

Without loss of generality, we may assume $\ell = 1$, so that applying $\pi$ deletes the first coordinate. 
To finish the proof, we apply the projection $\pi$ and note that the image contains all points of the form $(k_2 a_{\sigma(2)}, \ldots, k_n a_{\sigma(n)}) \in \mathbb{R}^{n-1}$ such that 
$\frac{1}{2} > a_1 > a_2 > \cdots > a_n > 0. $
Since $a_{\sigma(1)}$ is gone, we can just delete it from the inequality.  It follows that the image is an alcove in type $\widetilde{C}_{n-1}/C_{n-1}$, as desired.
\end{proof}

\begin{rem}\label{inducedmap} By considering the image of an alcove under the projection $\pi$ from \eqref{defpi}, Lemma \ref{alcove} tells us that $\pi$ induces a map $\C\to\widetilde{C}_{n-1}$. By abuse of notation, we denote this induced map by $\pi$ as well.
\end{rem}

\begin{lem}\label{fundamentalDomain}
The image of a fundamental region under $\pi$ is a fundamental region via the identification of $H_n^k$ with $\mathbb{R}^{n-1}$. 
\end{lem}
\begin{proof}
In type $\widetilde{C}_n$, the fundamental region in $\mathbb{R}^{n}$ is the region containing the points $(a_1, \ldots, a_n)$ where $- \frac{1}{2} \leq a_i \leq \frac{1}{2}$ for all $1 \leq i \leq n$.  As in the proof of the previous lemma, without loss of generality, we assume $\ell = 1$ so that applying $\pi$ deletes the first coordinate.  The image of the fundamental region is $(a_2, \ldots, a_n)$, where $- \frac{1}{2} \leq a_i \leq \frac{1}{2}$ for all $2 \leq i \leq n$, which is the fundamental region in type $\widetilde{C}_{n-1}$. 
\end{proof}

\begin{df}
Denote by $\mathcal{A}_{\circ}'$ the distinguished alcove on $H^k_n$ whose coroot lattice point is $\left \lceil \frac{k}{2n} \right \rceil \varepsilon_{\ell}$.  The alcove $\mathcal{A}_{\circ}'$ gets identified with the fundamental alcove in $\widetilde{C}_{n-1}$ under the projection $\pi$, which makes it convenient to establish a separate notation identifying this alcove.
\end{df}


\begin{lem}\label{lemming}
Exactly $k$ hyperplanes separate the fundamental alcove $\mathcal{A}_\circ$ and the alcove $\mathcal{A}_{\circ}'$.
\end{lem}
\begin{proof}

Recall that the coroot lattice point associated with $\mathcal{A}_\circ'$ is $\left\lceil\frac{k}{2n}\right\rceil\varepsilon_{\ell}$.  The abacus corresponding to this point has the lowest bead at level $\K$ in runner $\ell$ , at level $-\K$ in runner $N-\ell$ (where $N = 2n+1$, as in Definition~\ref{mirroredPermutation}), and at level zero elsewhere. We see that the value of the lowest bead $B$ in runner $\ell$ is $N\K+\ell-1$. Furthermore, there exists a unique bead $b$ on runner $\ell$ whose value lies in the range $[n+1,N+n]$. Note that the value labeling $b$ is $\ell$ if $\ell\geqslant n+1$ and $\ell+N$ otherwise. If $g$ denotes the number of gaps between $B$ and $b$, and $p$ denotes the number of beads whose values are greater than $N+n$, then
\begin{equation}
g=\begin{cases}(2n-1)\left(\K-1\right)+\ell-1 & \mbox{if}\quad \ell\geqslant n+1,\\ (2n-1)\left(\K-1\right) & \mbox{if}\quad \ell<n+1,\end{cases}
\end{equation}
and
\begin{equation}
p=\begin{cases}\K & \mbox{if}\quad \ell\geqslant n+1,\\ \K-1 & \mbox{if}\quad \ell<n+1.\end{cases}
\end{equation}

Using Corollary 8.1 in \cite{HanusaJones}, we know then that the length of the word corresponding to $\mathcal{A}_\circ'$ is 
\begin{equation}
\begin{cases}g+p=\left(2n-1\right)\left(\K-1\right)+\ell-1+\K & \mbox{if}\quad \ell\geqslant n+1,\\ g+p+(b-N)=\left(2n-1\right)\left(\K-1\right)+\K-1+(\ell+N)-N & \mbox{if}\quad \ell<n+1.\end{cases}
\end{equation}
Both of these are equal to $k$ because by the definition of $\ell$, we have $k=2n(\K-1)+\ell$. From this we conclude that any minimal length walk from $\mathcal{A}_\circ$ to $\mathcal{A}_\circ'$ passes through precisely $k$ hyperplanes. 
\end{proof}

We now develop several definitions which will be critical to the remaining geometric arguments in this section.  We will have occasion to separate out steps in an alcove walk which are parallel to a fixed hyperplane $H^k_n$ because these are the steps which survive after applying the projection $\pi$.

\begin{df}\label{perpstep} Given a pair of roots $\alpha$ and $\beta$ and two non-negative integers $m$ and $m'$, we call two affine hyperplanes $H_{\alpha,m}$ and $H_{\beta,m'}$ \textit{parallel} (resp. \textit{perpendicular}) if the roots $\alpha$ and $\beta$ are parallel (resp. perpendicular). Given a fixed hyperplane $H^k_n$, we call a step from an alcove $\mathcal{A}$ to an alcove $\mathcal{A}'$ in an alcove walk  \textit{perpendicular} if it is achieved by reflecting $\mathcal{A}$ across a hyperplane which is \emph{not} perpendicular to $H_n^k$. (The reason for this terminology is clear as illustrated by Figure~\ref{figPerpStep}.)  A step that is not perpendicular shall be called \textit{parallel}.  An alcove walk which consists entirely of perpendicular (resp. parallel) steps will be called a \emph{perpendicular walk} (resp. \emph{parallel walk}).
\end{df}

\begin{figure}
\begin{center}
\includegraphics[scale = .9]{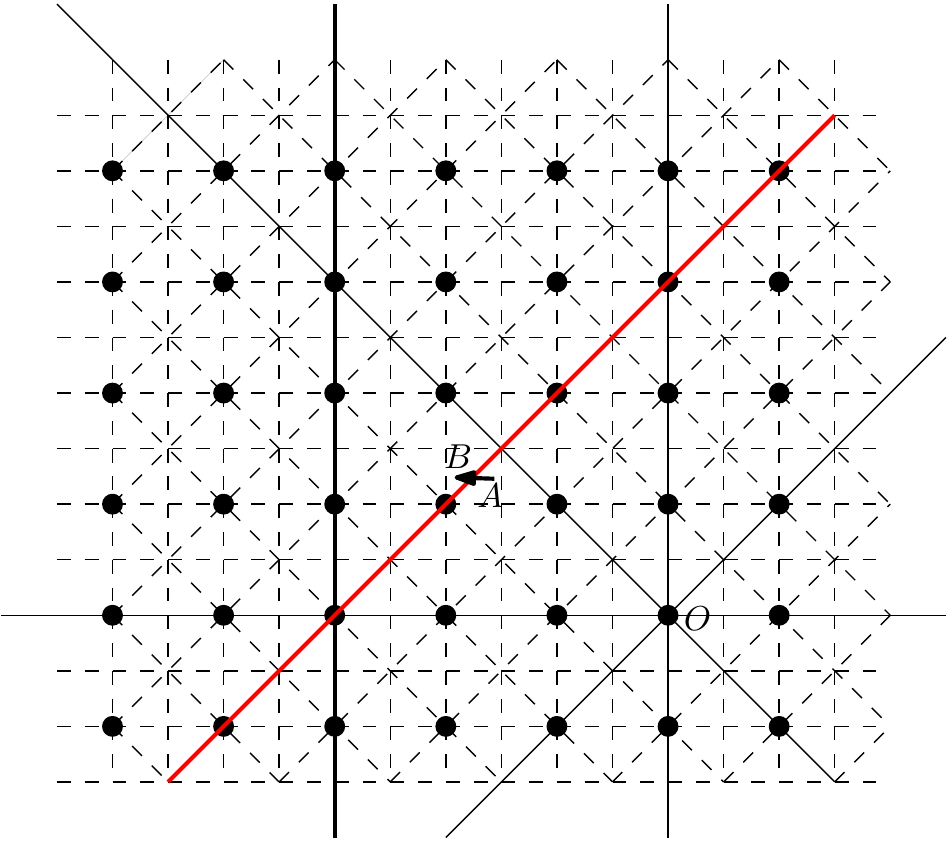}
\caption{A perpendicular step (black arrow) is obtained by reflecting over a hyperplane (red line) not perpendicular to $H_n^k$.}
\label{figPerpStep}
\end{center}
\end{figure}

Next, we will identify a certain subset of alcoves along the hyperplane $H^k_n$ which can be naturally identified with the alcove model one dimension lower.  In addition, this subset of ``good alcoves'' includes all of the distinguished alcoves on $H^k_n$ (\textit{i.e.}, those corresponding to minimal length coset representatives), which can therefore in turn be identified with the distinguished alcoves one dimension lower.  

As an auxiliary definition, we will start by identifying translates $T^k_n$ of the hyperplanes $H^k_n$ by a unit of $\frac{1}{2}$ toward the origin.  The good alcoves are then characterized as sharing a face with these translates, and so the good alcoves can be thought of as the alcoves along $H^k_n$ which are, in some sense, closest to the origin. We remind the reader that in this section we only discuss the case in which $\ell_2 = k\pmod{2n} \in \{ 1, \dots, n\}$.  There is an analogous definition for both the hyperplanes $T^k_n$ and good alcoves in the case in which $\ell_2 = k\pmod{2n} \in \{ n+1, \dots, 2n\}$, but all of the arguments are identical.

\begin{df}\label{goodAlcoves} \rm
Fix a hyperplane $H^k_n$.  Define the hyperplane $T_n^k = \left\{ v \in \mathbb{R}^n: \langle v, \varepsilon_{\ell} \rangle = \left \lceil \frac{k}{2n} \right \rceil - \frac{1}{2}\right\}$ to be the translate of $H^k_n$ shifted by a unit of $\frac{1}{2}$ toward the origin.   We call an alcove whose  coroot lattice point lies in ${\mathscr{R}}_{2n}^k$ a \emph{good alcove} if it shares an $(n-1)$-dimensional face with the hyperplane $T^k_n$.
\end{df}

\begin{prop}\label{perp steps}
Any minimal length alcove walk to a good alcove takes exactly $k$ steps which are perpendicular to $H_n^k$. 
\end{prop}
\begin{proof}
Observe from the symmetry of the hyperplane arrangement that steps in an alcove walk which are parallel to $H_n^k$ preserve the number of perpendicular steps required to cross the hyperplane $T_n^k$.  The result follows easily from this observation and Lemma~\ref{lemming}, since every good alcove can be obtained by reflecting $\mathcal{A}_{\circ}'$ across hyperplanes which are perpendicular to $H^k_n$.
\end{proof}

\begin{prop}
Distinguished alcoves whose coroot lattice points are in ${\mathscr{R}}_{2n}^k$ are good alcoves.  
\end{prop}
\begin{proof}

As discussed in Section \ref{rootlatticeptmodel}, we can partition the coroot lattice points lying in ${\mathscr{R}}_{2n}^k$ into $(2n-2)$ sets corresponding to the shifted Weyl chambers meeting ${\mathscr{R}}_{2n}^k$.  Within each shifted Weyl chamber, the position of the distinguished alcoves in each translate of the fundamental region is the same.  

Within a shifted Weyl chamber, every distinguished alcove whose lattice point lies on $H_n^k$ is clearly a translation of another distinguished alcove whose lattice point also lies on $H_n^k$ by a translation that is parallel to $H_n^k$, provided that there is more than one coroot lattice point in this shifted Weyl chamber.  Translation parallel to $H_n^k$ sends good alcoves to good alcoves since $T^k_n$ is parallel to $H^k_n$.  It follows that if one distinguished alcove in a given shifted Weyl chamber is good, then so are the rest of these alcoves. 

Note that the hyperplane $T^k_n$ intersects each translate of the fundamental region on $H^k_n$ in a full $(n-1)$-dimensional face, since $T^k_n$ is parallel to at least one face of the fundamental region.    Consider the distinguished alcove on $H^k_n$ within each shifted Weyl chamber whose coroot lattice point is closest to the origin.  Since there exists at least one alcove in the same coset which shares a face with $T^k_n$, this must be true for the distinguished alcove.  Indeed, any alcove which does not share a full $(n-1)$-dimensional face with $T^k_n$ has at least one more hyperplane separating it from the fundamental alcove, in which case the length of the corresponding word is longer than that of the alcove sharing a face with $T^k_n$.  Therefore, these distinguished alcoves are good, and the result follows by the previous paragraph.
\end{proof}

\begin{lem}\label{existGoodAlcoves}
For any alcove $\mathcal{B} \in \mathbb{R}^{n-1}$ whose coroot lattice point is in $\mathscr{R}_{2n-2}^{\leq \left(k - \left\lceil \frac{k}{n} \right\rceil \right) }$, there exists a good alcove $\mathcal{A}$ on $H^k_n$ such that $\pi(\mathcal{A}) = \mathcal{B}$.  
\end{lem}
\begin{proof}
Let $y \in \mathscr{R}_{2n-2}^{\leq \left(k - \left\lceil \frac{k}{n} \right\rceil \right)}$ be the coroot lattice point for the alcove $\mathcal{B}$.  By Proposition~\ref{domainHyperplane}, there exists a coroot lattice point $y' \in {\mathscr{R}}_{2n}^k$ such that $\pi(y') = y$.  Furthermore, from Lemma~\ref{fundamentalDomain}, we know that the translated fundamental region centered at $y'$ is mapped to the translated fundamental region centered at $\pi(y') = y$.  It suffices to show that there exists a good alcove with coroot lattice point $y'$ whose image under $\pi$ is $\mathcal{B}$.  

By Lemma~\ref{alcove}, the image of any good alcove with coroot lattice point $y'$ is an alcove with coroot lattice point $y$.  Moreover, it is clear that the image under $\pi$ of the intersection of $T_n^k$ and the translated fundamental region centered around $y'$ is the translated fundamental region centered around $y$.  It follows that there exists a good alcove $\mathcal{A}$ such that $\pi(\mathcal{A}) = \mathcal{B}$, as desired. 
\end{proof}

We now introduce ``clusters'' of alcoves as a way of partitioning alcoves into sets which allow maximal parallel movement without taking any perpendicular steps. A given alcove can reach any other alcove within its cluster via only parallel steps.  Conversely, to reach any alcove outside of its cluster, at least one perpendicular step must be taken. 

\begin{df}
Each alcove in $\widetilde{C}_n$ is contained in a \emph{cluster}, which is a set of alcoves which project down bijectively onto $H^k_n$ to alcoves in a translate of  the fundamental region of $\widetilde{C}_{n-1}$; see Figure~\ref{figWLL1}. All alcoves within a cluster are reachable from each other via a parallel walk, as defined in Definition \ref{perpstep}. These clusters are centered either at a coroot lattice point or at a coroot lattice point translated by $(1/2,\ldots, 1/2)$.  Equivalently, clusters are the sets of $(n-1)2^{n-1}$ alcoves which all share a face with the same hyperplane parallel to $H^k_n$, but also have a vertex in common which does not lie on $H^k_n$.
\end{df}

Taking steps within a cluster allows movement parallel to $H_n^k$ up to a unit in each direction without taking a perpendicular step. Two hyperplanes $H_{\alpha, k}$ and $H_{\alpha,m}$ are \emph{adjacent} if the difference $|k-m|$ is minimal.  Since all alcoves in a cluster have a face which lies in a hyperplane parallel to $H_n^k$, between any two adjacent hyperplanes parallel to $H_n^k$ we can make exactly one forward perpendicular step. 

\begin{lem}[Walk Lifting Lemma] \label{walk}
Let $\mathcal{A}$ be a good alcove  whose coroot lattice point lies in ${\mathscr{R}}_{2n}^k$.  If a minimal length alcove walk from $\pi(\mathcal{A}_{\circ})$ to $\pi(\mathcal{A})$ has length $s$, then any minimal length alcove walk from $\mathcal{A}_{\circ}$ to $\mathcal{A}$ has length $k + s$. 
\end{lem}

\begin{proof}
For any minimal length alcove walk from $\mathcal{A}_{\circ}$ to $\mathcal{A}$ of length $\ell$, consider its projection to the hyperplane $T_n^k$, which is the hyperplane parallel to $H_n^k$ that contains a face of $\mathcal{A}$, as described in Definition~\ref{goodAlcoves}.  This projection deletes all the steps in the walk which are perpendicular to $H_n^k$.  By Proposition~\ref{perp steps}, this projection is an alcove walk from $\pi(\mathcal{A}_{\circ})$ to $\pi(\mathcal{A})$ of length $\ell - k$.  

Now suppose that any minimal length alcove walk from $\pi(\mathcal{A}_{\circ})$ to $\pi(\mathcal{A})$ has length $s$.  The previous observation then says that $\ell - k \geq s$, or equivalently $\ell \geq k + s$.  To prove that in fact $\ell = k + s$, it suffices to show that there exists a minimal length alcove walk from $\pi(\mathcal{A}_{\circ})$ to $\pi(\mathcal{A})$ that can be lifted to a walk from $\mathcal{A}_{\circ}$ to $\mathcal{A}$, containing only parallel steps to $H_n^k$ as well as perpendicular steps which move the centroid of the alcove closer to the hyperplane $H^k_n$.  For ease of reference, we refer to this latter type of step as a \emph{forward} perpendicular step.  As long as the walk from $\mathcal{A}_{\circ}$ to $\mathcal{A}$ contains only forward perpendicular steps and parallel steps to $H^k_n$, it will contain exactly $k$ perpendicular steps to $H_n^k$ by Proposition \ref{perp steps}.

We now index hyperplanes $T_n^k$ by levels, based on their distance from the origin.  Define the \emph{level} of the hyperplane $T^k_n$ to be the ceiling of the minimal distance from the origin to any point on $T^k_n$, which equals $\lceil \frac{k}{2n} \rceil$.  If a good alcove $\mathcal{A}$ shares a face with a level $r$ hyperplane $T^k_n$, we say that $\mathcal{A}$ is on level $r$.  We will prove the claim by using induction on the level of the hyperplane in which the good alcove $\mathcal{A}$ is located. 

We first consider the case where the good alcove $\mathcal{A}$ shares a face with a level $r=1$ hyperplane $T^k_n$. We need to show there is a path from $\mathcal{A}_\circ$ to $\mathcal{A}$ which takes only parallel and forward perpendicular steps.  Since $\mathcal{A}$ is on level $1$, then the coroot lattice point for $\mathcal{A}$ is between $-1$ and $1$ in each coordinate, and so either $\pi(\mathcal{A})=\pi(\mathcal{A}')$ for some alcove $\mathcal{A}'$ in the fundamental region, or $\mathcal{A}$ lies outside the fundamental region by at most one unit in each direction. In the first case, it is clear that $\mathcal{A}$ is reachable from $\mathcal{A}_{\circ}$ via a perpendicular walk. Indeed, the walk will consist of a sequence of forward perpendicular steps entirely within the fundamental region followed by a single forward perpendicular step to reach the good alcove on $T^k_n$.  Now suppose that $\pi(\mathcal{A})=\pi(\mathcal{A}')$ for some alcove $\mathcal{A}'$ which is not in the fundamental region. In this case, we can assume that $\mathcal{A}'$ lies in the same cluster as an alcove $\mathcal{A}_f$ which is in the fundamental region.  The walk from $\mathcal{A}_{\circ}$ to $\mathcal{A}$ should then begin by taking a perpendicular walk to $\mathcal{A}_f$ within the fundamental region (if $\mathcal{A}_f$ does not equal $\mathcal{A}_{\circ}$), then taking the parallel walk to $\mathcal{A}'$ within the cluster, followed by another perpendicular walk from $\mathcal{A}'$ to $\mathcal{A}$ as in the previous case.  In either scenario, we can lift the walk from $\pi(\mathcal{A}_{\circ})$ to $\pi(\mathcal{A})$ to an alcove walk from $\mathcal{A}_{\circ}$ to the good alcove $\mathcal{A}$ on level $1$.

For the induction hypothesis, suppose all the good alcoves at level $r$ satisfy the Walk Lifting Lemma.  Now, consider a good alcove $\mathcal{A}$ located at level $r + 1$.  If this good alcove has the same projection onto $T^k_n$ as a good alcove $\mathcal{A}'$ located at level $r$, then for the first part of the  walk from $\mathcal{A}_{\circ}$ to $\mathcal{A}$, we can just take the lift of the minimal alcove walk from $\pi(\mathcal{A}_{\circ})$ to $\pi(\mathcal{A}')$ given by the induction hypothesis (as pictured in Figure \ref{figWLL1}).  Once we reach $\mathcal{A}'$, we can then take a (forward) perpendicular walk to level $r+1$ to obtain a minimal length alcove walk from $\mathcal{A}_{\circ}$ to $\mathcal{A}$.  

\begin{figure}[htp]
\begin{center}
\includegraphics[scale = 1]{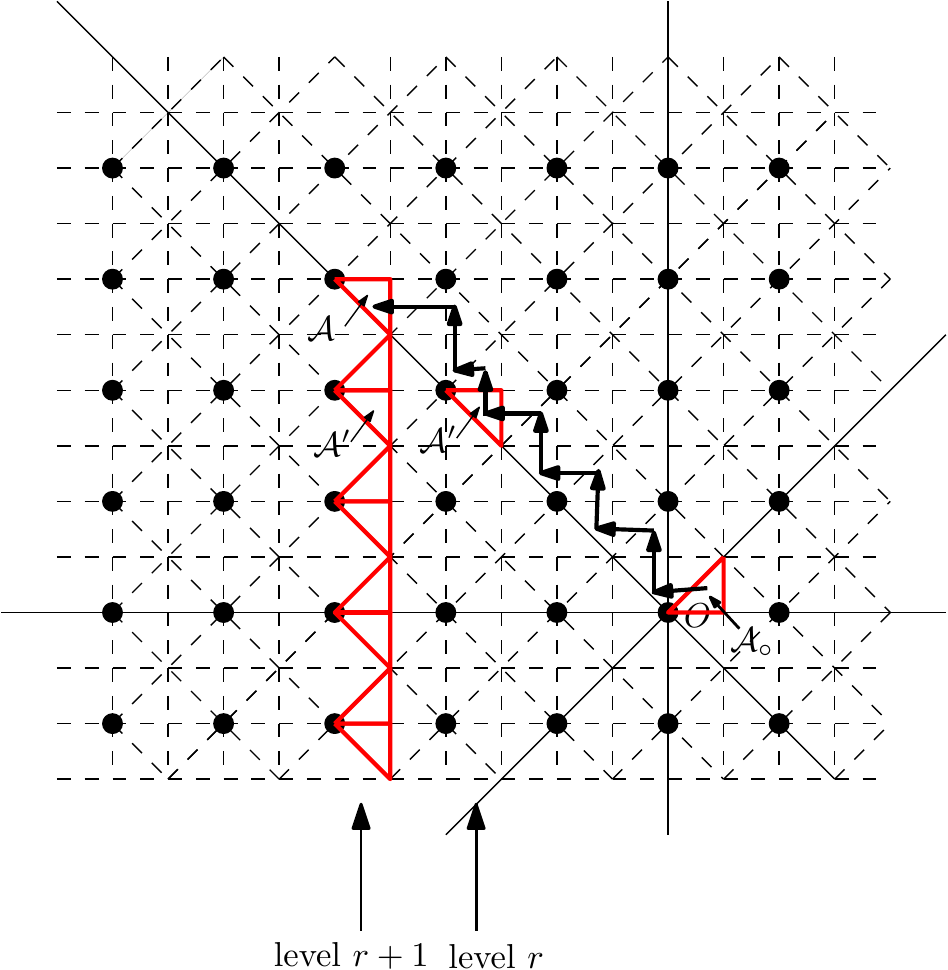}
\caption{Going from a good alcove at level $r$ to a good alcove at level $r+1$.}
\label{figWLL1}
\end{center}
\end{figure}

Finally, suppose the good alcove at level $r+1$ does not have the same projection onto $H^k_n$ as any alcove at level $r$.  Consider any minimal alcove walk from $\pi(\mathcal{A}_{\circ})$ to $\pi(\mathcal{A})$.  Let $\mathcal{A}'$ be the alcove on $T^k_n$ closest to $\mathcal{A}$ which does have the same projection as an alcove at level $r$, which by an abuse of notation we also call $\mathcal{A}'$ (see Figure \ref{figWLL1}).  We construct a lift of the walk from $\pi(\mathcal{A}_{\circ})$ to $\pi(\mathcal{A})$ as follows.  First lift the path from $\pi(\mathcal{A}_{\circ})$ to $\pi(\mathcal{A'})$ to a walk from $\mathcal{A}_{\circ}$ to the alcove $\mathcal{A}'$ at level $r$ by the induction hypothesis.  Similar to the base case of the induction, we construct the final part of the lift by taking a parallel walk within the cluster containing $\mathcal{A}'$ at level $r$, followed (if necessary) by a single perpendicular step to move to the next cluster, remaining between the same two adjacent hyperplanes parallel to $H^k_n$ which bound $\mathcal{A}'$.  Within this next cluster (at the latest), we will reach an alcove which has the same projection onto $H^k_n$ as $\mathcal{A}$.  From that point, take the perpendicular walk directly to $\mathcal{A}$.  Altogether, we have a walk from $\mathcal{A}_{\circ}$ to $\mathcal{A}$ consisting entirely of forward perpendicular and parallel steps, proving the claim for level $r + 1$.  By induction, it follows that all good alcoves satisfy the Walk Lifting Lemma. 
\end{proof}

Equipped with the Walk Lifting Lemma, we are now prepared to derive our main geometric results.  Let $\mathscr{X}_{2n}^k$ denote the distinguished alcoves whose associated coroot lattice points are elements of ${\mathscr{R}}_{2n}^k$ and let $\mathscr{X}_{2n}^{\leqslant k-\lceil\frac{k}{n}\rceil}$ be defined similarly.

\begin{thm}\label{phi(w)} The following diagram commutes:

\[
\xymatrix{
\mathscr{S}_{2n}^k \ar[r]^{F_\mathscr{X} } \ar[d]^{\Phi_n^k} & \mathscr{X}_{2n}^k \ar[d]^{\pi} \\ 
\mathscr{S}_{2n-2}^{\leq \left(k - \left\lceil \frac{k}{n} \right \rceil \right)} \ar[r]^{F_\mathscr{X}} & \mathscr{X}_{2n-2}^{\leq \left(k - \left\lceil \frac{k}{n} \right\rceil \right)} 
}
\]

\end{thm}
\begin{proof} 
Start with a core $S \in \mathscr{S}_{2n}^k.$  By the commutative diagram~\eqref{CoresToAbaci}, the coroot lattice point associated to the image of $F_{\mathscr{X}}(S)$ is precisely the same as the coroot lattice point associated to the alcove $F_{\mathscr{X}}(\Phi_n^k(S))$.  To complete the theorem, it suffices to show that $\pi$ sends distinguished alcoves to distinguished alcoves. 

To see this, suppose by contradiction that $\mathcal{A}\in\mathscr{X}_{2n}^k$ is an alcove such that $\pi(\mathcal{A})$ is not distinguished.  Let $s$ be the length of the minimum length alcove walk to $\pi(\mathcal{A})$.  There exists an alcove $\mathcal{B}$ with the same coroot lattice point as $\pi(\mathcal{A})$ to which there exists a minimal length walk of length $r < s$.  By Lemma~\ref{existGoodAlcoves}, there exists a good alcove $\mathcal{A}'$ in the same coset as $\mathcal{A}$ so that $\pi(\mathcal{A}') = \mathcal{B}$.  Since both $\mathcal{A}$ and $\mathcal{A}'$ are good, by the Walk Lifting Lemma there exist minimal length alcove walks to $\mathcal{A}$ and $\mathcal{A}'$ of lengths $k+s$ and $k+r$, respectively.  Since $k+s > k+r$, this is a contradiction to the fact that $\mathcal{A}$ is a distinguished alcove.  It follows that $\pi(\mathcal{A})$ must be distinguished. 
\end{proof}

\begin{thm}\label{MainTheorem3}
Let $w$ be a minimal length coset representative for $\widetilde{C}_n/C_n$ such that the associated core has first part $k$; \textit{i.e.}, suppose that the alcove $\mathcal{A}_w \in \mathscr{X}^k_{2n}$.  If 
$\mathcal{A}_{\circ} = \mathcal{A}^1\to\cdots\to\mathcal{A}^r = \mathcal{A}_w$
is a minimal length alcove walk for $w$, then 
\begin{equation}\label{1.6}
\pi(\mathcal{A}^1)\to\cdots\to\pi(\mathcal{A}^r)
\end{equation}
is an alcove walk for $\Phi_n^k(w)$. 
Moreover, if one removes all repeated instances of alcoves in \eqref{1.6}, then the resulting walk is a minimal length alcove walk for $\Phi_n^k(w)$.
\end{thm}

\begin{proof}
By Lemmas \ref{alcove} and \ref{fundamentalDomain}, we know that $\pi(\mathcal{A}_\circ)$ is the fundamental alcove for $\widetilde{C}_{n-1}$.  In addition, Theorem \ref{phi(w)} shows that $\pi(\mathcal{A}_w)$ corresponds to $\Phi_n^k(w)$.  Moreover,  $\pi$ takes adjacent alcoves to adjacent alcoves, so \eqref{1.6} is an alcove walk from the identity alcove in $\widetilde{C}_{n-1}$ to the alcove corresponding to $\Phi_n^k(w)$.  The removal of the repeated instances of alcoves still clearly yields an appropriate alcove walk under $\pi$.  Moreover, the Walk Lifting Lemma says that this must in fact be a minimal length walk, since any shorter walk would lift to a shorter original walk to $\mathcal{A}_w$.
\end{proof}

\begin{cor}\label{MainTheorem2a}
Let $w$ be a minimum length coset representative for $\widetilde{C}_n/C_n$. Then 
\begin{equation}\ell_{\widetilde{C}_n}(w)-\ell_{\widetilde{C}_{n-1}}(\Phi_n^k(w))=k.\end{equation}
\end{cor}
\begin{proof}
Start with a minimal length alcove walk
$$\mathcal{A}_\circ=\mathcal{A}^1\to\cdots\to\mathcal{A}^{s+1}=F_\mathscr{X}(w),$$
where $s = \ell_{\widetilde{C}_n(w)}$.  Then 
$$\pi(\mathcal{A}_\circ)\to\cdots\to\pi(F_\mathscr{X}(w))$$
is an alcove walk for $F_\mathscr{X}(\Phi_n^k(w))$.  Note that in this alcove walk there are exactly $k$ repeated instances of alcoves, corresponding to the $k$ perpendicular steps to $H_n^k$ guaranteed by Proposition \ref{perp steps}.  We may delete these repeated instances of alcoves to get an alcove walk of length $s - k$.  This alcove walk is minimal, for if it is not, then there exits an alcove walk of length $r < s - k$, which, by the Walk Lifting Lemma may be lifted to an alcove walk from $\mathcal{A}_{\circ}$ to $F_{\mathscr{X}}(w)$ of length $r + k < s$, contradicting the fact that $\ell(w) = s$.  Therefore $\ell_{\widetilde{C}_{n-1}}(\Phi_n^k(w)) = \ell_{\widetilde{C}_n}(w) - k$ and the claim follows. 
 \end{proof}


\section{Action of $\Phi_n^k$ On Reduced Words} \label{sectionAction}

In this section, we will describe the action of the map $\Phi_n^k$ on reduced words of $\widetilde{C}_n/C_n$.  We remark that analogous claims for type $A$ were stated without proof in \cite{BJV}, and so we include all of the details here.  Let $s_{i_1} \cdots s_{i_{\ell}}$ be a reduced word in $\widetilde{C}_n/C_n$ which corresponds to the abacus $\mathfrak{a} = (a_1, \ldots, a_n, -a_n, \ldots, -a_1)$.  From the correspondence between reduced words and abacus diagrams, we can see that $s_{i_{\ell}}\cdots s_{i_1} (\mathfrak{a}) = (0, \ldots, 0)$, is the abacus diagram corresponding to the identity element in $\widetilde{C}_n/C_n$. 

From $s_{i_1} \cdots s_{i_{\ell}}$, we can build a reduced word for $\Phi_n^k(\mathfrak{a})$ in the following $\ell$ steps.  First, consider the action of $s_{i_1}$ on the abacus $\mathfrak{a}$.  Using the results of Section 3.2 in \cite{HanusaJones}, if $s_{i_1}$ acts on $\mathfrak{a}$ by changing the position of the largest runner, set $t_1 = 1 \in \widetilde{C}_{n-1}$.  Otherwise, if applying $s_{i_1}$ to $\mathfrak{a}$ does not change the position of the largest runner, then there exists a unique generator $s_{i_1'} \in \widetilde{C}_{n-1}$, where $i_1' = i_1$ or $i_1 - 1$, depending on the positions of the runners being changed relative to the largest runner, such that $\Phi_n(s_{i_1}\mathfrak{a}) = s_{i_1'} \Phi_n(\mathfrak{a})$.  In this case, set $t_1 = s_{i_1'}$.

At step $r$ for $1 \leq r \leq \ell$, consider the action of $s_{i_r}$ on the abacus $s_{i_{r-1}} \cdots s_{i_1} ( \mathfrak{a})$.  As with the first step, if $s_{i_r}$ changes the position of the largest runner, set $t_r = 1 \in \widetilde{C}_{n-1}$.  Otherwise, there exists a unique generator $s_{i_r'} \in \widetilde{C}_{n-1}$ such that 
$$\Phi_n(s_{i_r} s_{i_{r-1}} \cdots s_{i_1}\cdot \mathfrak{a}) = s_{i_r'} \Phi_n(s_{i_{r-1}} \cdots s_{i_1} \cdot \mathfrak{a}) = s_{i_r'} t_{r-1} \cdots t_1 \cdot \Phi_n(\mathfrak{a}).  $$
In this case, set $t_r = s_{i_r'}$.  

From our construction, we obtain the following commutative diagram for all $1 \leq r \leq \ell$: 
$$\xymatrix{
s_{i_{r-1}} \cdots s_{i_1} \cdot \mathfrak{a} \ar[r]^-{\Phi_n} \ar[d]^{s_{i_r}} &t_{r-1} \cdots t_1 \cdot \Phi_n(\mathfrak{a}) \ar[d]^{t_r} \\
s_{i_r} s_{i_{r-1}} \cdots s_{i_1} \cdot \mathfrak{a} \ar[r]^-{\Phi_n} &t_r t_{r-1} \cdots t_1 \cdot \Phi_n(\mathfrak{a})
}$$
It follows that $t_1 \cdots t_{\ell}$ is a word corresponding to the abacus $\Phi_n(\mathfrak{a})$.  By Lemma~\ref{lemLambda1}, since $s_{i_1} \cdots s_{i_{\ell}}$ is a reduced word for the abacus $\mathfrak{a}$, there are exactly $\lambda_1$ instances when the generator $s_{i_j}$ changes the largest runner.  In fact, it will either move the largest runner to the left by one, or, if the largest runner is the first runner, decrease the level of the largest runner by one and move it to runner $2n$.  In other words, the word $t_1 \cdots t_{\ell}$ has a length of exactly $\ell - \lambda_1$.  

\begin{prop}
The word $t_1 \cdots t_{\ell}$ is a reduced word corresponding to the abacus $\Phi_n(\mathfrak{a})$. 
\end{prop}
\begin{proof}
It suffices to show that the length of a reduced word for $\Phi_n(\mathfrak{a})$ is $\ell - \lambda_1$.  We will proceed via contradiction.  Suppose $s_{i_1} \cdots s_{i_{p}}$ is a reduced word for $\Phi_n(\mathfrak{a})$, with $p < \ell - \lambda_1$.  We will construct a word for $\mathfrak{a}$ that has length $p + \lambda_1 < \ell$.  

By definition, $s_{i_{p}} \cdots s_{i_1} \cdot \Phi_n(\mathfrak{a}) = (0, \ldots, 0)$ corresponds to the abacus representing the identity element of the quotient $\widetilde{C}_{n-1}/C_{n-1}$.  We will construct a word for $\mathfrak{a}$ as follows.  For the first step, consider the action of $s_{i_1}$ on the abacus $\Phi_n(\mathfrak{a})$.  There are two cases to consider: 
\begin{enumerate}[label= (\roman{*})]
\item Suppose $1 \leq i_1 \leq n-1$.  The generator $s_{i_1}$ swaps two runners (along with their symmetric runners).  Consider the positions of the two runners in the abacus $\mathfrak{a}$.  If they are adjacent in $\mathfrak{a}$, then there exists a generator $s_{i_1'} \in \widetilde{C}_n$ such that $s_{i_1'}$ switches these two runners in the abacus $\mathfrak{a}$.  If this is the case, set $w_1 = s_{i_1'}$.  On the other hand, if the two runners are not adjacent in $\mathfrak{a}$, then they must be separated by the longest runner or the symmetric runner of the longest runner. In this case, set $w_1 = s_{i_1''}s_{i_1'}$, where $s_{i_1'}$ is the generator that moves the longest runner to the left by one position, and $s_{i_1''}$ is the generator that swaps the two desired runners.\\

\item Suppose $i_1 = 0$.  The generator $s_0$ swaps the first runner and the last runner, increases the first runner by one level, and decreases the last runner by one level.  Consider the position of this first runner in the abacus $\mathfrak{a}$.  If this runner is the first runner, set $w_1 = s_0$.  On the other hand, if this runner is not the first runner, then the first runner must be the longest runner or the symmetric runner of the longest runner.  If the first runner is the longest runner, set $w_1 = s_0 s_1 s_0$.  If the first runner is the symmetric runner of the longest runner, then set $w_1 = s_0 s_1$.  
\end{enumerate}

It can be checked from our construction that $w_1 \mathfrak{a} = s_{i_1} \Phi_n(\mathfrak{a})$.  \\

At step $r$ for $1 \leq r \leq p$, suppose we have inductively constructed $w_j$ for all $1 \leq j \leq r-1$ such that $\Phi_n(w_{r-1} \cdots w_1 \cdot \mathfrak{a}) = s_{i_{r-1}} \cdots s_{i_1} \cdot \Phi_n(\mathfrak{a})$.  Consider the action of $s_{i_r}$ on the abacus $s_{i_{i-1}} \cdots s_{i_1} \cdot \Phi_n(\mathfrak{a})$.  We construct $w_r$ in the same manner as was described above so that the following diagram commutes: 
$$\xymatrix{
w_{r-1} \cdots w_1 \cdot \mathfrak{a} \ar[r]^-{\Phi_n} \ar[d]^{w_r} &s_{i_{r-1}} \cdots s_{i_1} \cdot \Phi_n(\mathfrak{a}) \ar[d]^{s_{i_r}} \\ 
w_r w_{r-1} \cdots w_1 \cdot \mathfrak{a} \ar[r]^-{\Phi_n} & s_{i_r} s_{i_{r-1}} \cdots s_{i_1} \Phi_n(\mathfrak{a}) 
}$$

To finish our construction, notice that after step $p$, we have that $\Phi_n(w_p w_{p-1} \cdots w_1 \cdot \mathfrak{a}) = (0, \ldots, 0)$ is the abacus corresponding to the identity element in $\widetilde{C}_{n-1}/C_{n-1}$.  However, $w_p w_{p-1} \cdots w_1 \cdot \mathfrak{a}$ is not necessarily the abacus corresponding to the identity element in $\widetilde{C}_n/C_n$.  Let $w_{p+1}$ be the shortest word so that $w_{p+1} w_p \cdots w_1 \cdot \mathfrak{a} = (0, \ldots, 0)$ is the abacus corresponding to the identity element in $\widetilde{C}_n/C_n$.  

From our construction, the word $(w_{p+1} \cdots w_1)^{-1} = w_1^{-1} \cdots w_{p+1}^{-1}$ is a word corresponding to the abacus $\mathfrak{a}$, and it has length $\ell(s_{i_1} \cdots s_{i_p}) + \lambda_1 = p + \lambda_1$.  This is because every extra generator we added moves the largest runner to the left, or moves it from the first runner to the last runner while decreasing its level by one.  Since $p + \lambda_1 < \ell$, the length of the word corresponding to the abacus $\mathfrak{a}$, we have arrived at a contradiction.  Therefore, the length of $\Phi_n(\mathfrak{a})$ is $\ell - \lambda_1$, as desired. 
\end{proof}

\begin{exam}\rm
The canonical reduced word for the abacus $\mathfrak{a} = (2, 1, -1, 1, -1, -2)$ can be found as described in Section \ref{S:canonicalwords} to be $s_0 s_1 s_3s_2s_3s_0s_1s_2s_0 s_1 s_0$.  Using the procedure outlined in this section, a reduced word for $\Phi_n(\mathfrak{a}) = (1, -1, 1, -1)$ is $1 \cdot 1 \cdot s_2 \cdot 1 \cdot 1 \cdot s_0 s_1 \cdot 1 \cdot  s_0 \cdot 1 \cdot 1 = s_2 s_0 s_1 s_0.$
\begin{eqnarray}
(2, 1, -1, 1, -1, -2) &\stackrel{s_0}{\to}& (-1, 1, -1, 1, -1, 1) \stackrel{s_1}{\to} (1, -1, -1, 1, 1, -1) \stackrel{s_3}{\to} (1, -1, 1, -1, 1, -1)  \nonumber \\ 
&\stackrel{s_2}{\to}& (1, 1, -1, 1, -1, -1) \stackrel{s_3}{\to} (1, 1, 1, -1, -1, -1) \stackrel{s_0}{\to} (0, 1, 1, -1, -1, 0) \nonumber\\ 
&\stackrel{s_1}{\to}& (1, 0, 1, -1, 0, -1) \stackrel{s_2}{\to} (1, 1, 0, 0, -1, -1) \stackrel{s_0}{\to} (0, 1, 0, 0, -1, 0)  \nonumber \\ 
&\stackrel{s_1}{\to}& (1, 0, 0, 0, 0, -1) \stackrel{s_0}{\to} (0, 0, 0, 0, 0, 0) \nonumber \\
( 1, -1, 1, -1) &\stackrel{1}{\to}& (1, -1, 1, -1) \stackrel{1}{\to} (1, -1, 1, -1) \stackrel{s_2}{\to} (1, 1, -1, -1)  \nonumber \\ 
&\stackrel{1}{\to} &(1, 1, -1, -1) \stackrel{1}{\to} (1, 1, -1, -1) \stackrel{s_0}{\to} (0, 1, -1, 0) \nonumber \\  
&\stackrel{s_1}{\to}& (1, 0,  0, -1) \stackrel{1}{\to} (1, 0, 0, -1) \stackrel{s_0}{\to} (0, 0, 0,  0) \nonumber \\ 
&\stackrel{1}{\to} &(0, 0, 0, 0) \stackrel{1}{\to} (0, 0, 0, 0) \nonumber
\end{eqnarray}
\end{exam}

\begin{cor}[The map $\Phi_n^k$ decreases  length by exactly $k$]\label{MainTheorem2b}
 For any abacus $\mathfrak{a}$ corresponding to a symmetric core partition in $\mathcal{S}_{2n}^k$, we have that $\ell(\Phi_n^k(\mathfrak{a})) = \ell(\mathfrak{a})-k$. Here, by $\ell(\mathfrak{a})$ we mean the length of the canonical reduced word corresponding to the abacus $\mathfrak{a}$.
\end{cor}


\section{The Map $\Phi_n^k$ preserves Bruhat order} \label{sectionProperties}

Fix a hyperplane $H_n^k$.  In this section, we will show that under the identification of $H_n^k$ with $\mathbb{R}^{n-1}$, strong Bruhat order is preserved when alcoves having coroot lattice points lying on $H_n^k$ are projected onto $H_n^k$.  

\begin{thm}[Theorem 5.11 in \cite{HanusaJones}]\label{coreContainment}
Let $x, y \in \widetilde{W}/W$.  Then $x \geq_B y$ if and only if the core diagram for $x$ contains the core diagram for $y$.
\end{thm}

We begin with a lemma on abacus diagrams, which uses the above fact that strong Bruhat order is equivalent to containment of cores.  By an abuse of notation, we will use the same letters to denote both the abacus diagram and the core partition associated to an element of $\widetilde{C}_n/C_n$.  

\begin{lem}\label{kthBead}
Let $x$ and $y$ be elements in $\widetilde{C}_n/C_n$.  Define the the $k^{\text{th}}$ highest bead in an abacus diagram by reading along levels right to left, starting with the highest level.  Then $x \geq_B y$ if and only if for all $k \geq 1$, the $k^{\text{th}}$ highest bead in $x$ is as high as the $k^{\text{th}}$ highest bead in $y$.  
\end{lem}

\begin{proof}
The number of gaps smaller than the $k^{\text{th}}$ highest bead in $x$ and $y$ is the length of the $k^{\text{th}}$ row in the core partitions of $x$ and $y$.  Let the highest bead in $x$ be located on runner $i_x$ at level $\ell_x$, and the highest bead for $y$ on runner $i_y$ at level $\ell_y$.  By Lemma~\ref{lemLambda1}, the number of gaps in $x$ that are smaller than the highest bead in $x$ equals $2n (\ell_x - 1) + i_x$, and the number of gaps in $y$ that are smaller than the highest bead in $y$ equals $2n(\ell_y -1) + i_y$.  

First suppose $x \geq_B y$.  It is clear from the previous paragraph that the highest bead in $x$ is as high as the highest bead in $y$.  Moreover, there are $2n(\ell_x - \ell_y) + (i_x - i_y)$ more gaps in $x$ that are smaller than the first highest bead than there are in $y$.  This number is exactly the number of positions that are higher than the highest bead in $y$, but not higher than the highest bead in $x$.  It can be easily checked that since the core partition for $x$ contains the core partition of $y$, the number of gaps smaller than the $k^{\text{th}}$ bead in $x$ is at least the number of gaps smaller than the $k^{\text{th}}$ bead in $y$, and so the $k^{\text{th}}$ bead in $x$ must be as high as the $k^{\text{th}}$ bead in $y$, as desired. 

To prove the converse, note that from the discussion above, if the $k^{\text{th}}$ highest bead in $x$ is as high as the $k^{\text{th}}$ highest bead in $y$, then the number of gaps that are smaller than the $k^{\text{th}}$ bead in $x$ is at least the number of gaps that are smaller than the $k^{\text{th}}$ highest bead in $y$.  It follows that the core partition of $x$ contains the core partition of $y$, and so $x \geq_B y$. 
\end{proof}

\begin{thm}[Bruhat order is preserved]\label{BruhatPreserved}
Let $x$ and $y$ be elements in $\widetilde{C}_n/C_n$ whose associated coroot lattice points lie on $H_n^k$.  Identify $H_n^k$ with $\mathbb{R}^{n-1}$ and let $\pi$ be the projection map onto $H_n^k$.  Then $x \geq_B y$ if and only if $\pi(x) \geq_B \pi(y)$. 
\end{thm}

\begin{proof}
Our proof will use abacus diagrams.  Suppose $x \geq_B y$.  By Lemma~\ref{kthBead}, the $k^{\text{th}}$ highest bead in $x$ is as high as the $k^{\text{th}}$ highest bead in $y$.  When we apply the projection map $\pi$, we delete the same beads in both abaci.  When we record the beads in an abacus diagram listed from highest to lowest $(b_1, b_2, b_3, \dots)$, define the rank of bead $b_i$ to be its position $i$ in this list.  Consider the ranks of the deleted beads in $x$ and $y$.  Since the $k^{\text{th}}$ highest bead in $x$ is at least as high as the $k^{\text{th}}$ highest bead in $y$, the rank of a deleted bead in $x$ is no higher than its rank in $y$.  After deleting those beads, the $k^{\text{th}}$ highest bead in $\pi(x)$ is still as high as the $k^{\text{th}}$ highest bead in $\pi(y)$ for all $k \geq 1$.  It follows that $\pi(x) \geq \pi(y)$. 

Conversely, suppose $\pi(x) \geq_B \pi(y)$.  By Lemma~\ref{kthBead} again, the $k^{\text{th}}$ highest bead in $\pi(x)$ is as high as the $k^{\text{th}}$ highest bead in $\pi(y)$ for all $k \geq 1$.  To obtain $x$ and $y$ from $\pi(x)$ and $\pi(y)$, we are inserting the same beads for both abaci.  Given a bead to be inserted, its rank after its insertion into $\pi(x)$ is no higher than its rank after its insertion into $\pi(y)$.  Therefore, after the addition of all the beads, the $k^{\text{th}}$ highest bead in $x$ is still as high as the $k^{\text{th}}$ highest bead in $y$.  It follows that $x \geq_B y$, as desired. 
\end{proof}

\begin{figure}[htp]
\begin{center}
\subfigure[$x = (1, 2, -2, 2, -2, -1)$]{\includegraphics[scale = 0.6]{Cabacus2}} \hspace{0.3in}
\subfigure[$y = (1, 0, -2, 2, 0, -1)$]{\includegraphics[scale = 0.6]{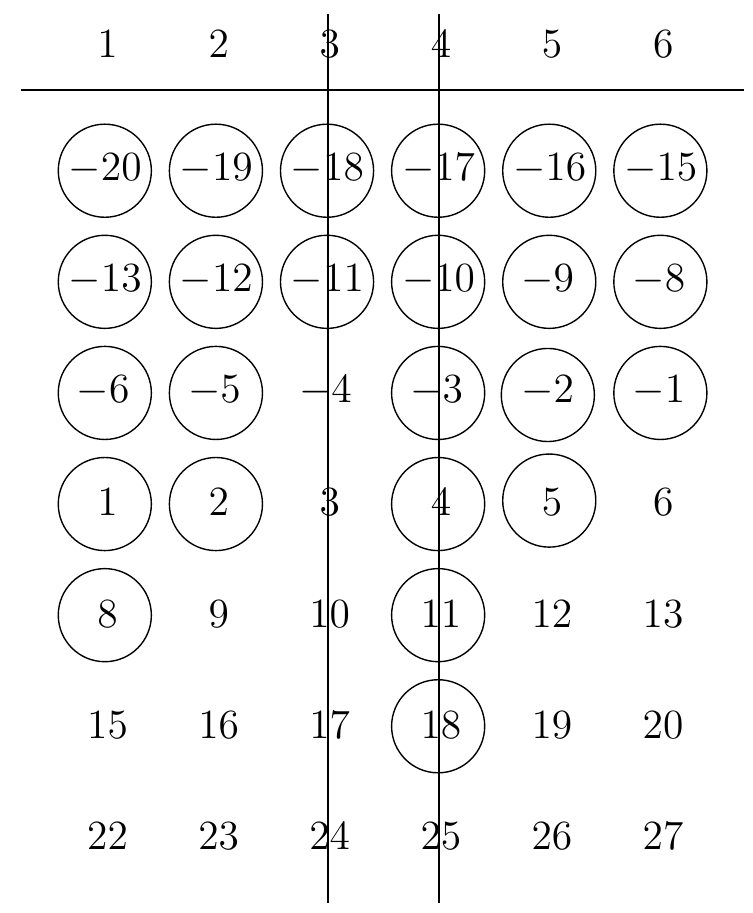}} 
\caption{The map on abaci preserves Bruhat order}
\label{figBruhatOrderIf}
\end{center}
\end{figure}

\begin{rem}
We remark that Theorem~\ref{BruhatPreserved} also generalizes to affine type $A$.  Indeed, the corresponding statement of Lemma~\ref{kthBead} for affine type $A$ is a consequence of the version of the lemma proved here in affine type $C$.  Since Theorem~\ref{coreContainment} holds for affine type $A$ as well, our proof of Theorem~\ref{BruhatPreserved}, used word for word, gives a proof of the fact that $\Phi_n^k$ also preserves strong Bruhat order in affine type $A$. 
\end{rem}

\begin{exam} \rm
As shown in Figure~\ref{figBruhatOrderIf}, the beads in $x = (1, 2, -2, 2, -2, -1)$ are numbered from highest to lowest by $(\underline{18}, 16, \underline{11}, 9, 8, \underline{4}, 2, 1, -1, \underline{-3}, -5, -6, \ldots)$, and the beads in $y = (1, 0, -2, 2, 0, -1)$ are numbered $(\underline{18}, \underline{11}, 8, 5, \underline{4}, 2, 1, -1, -2, \underline{-3}, -5, -6, -8, \ldots)$, where the beads to be deleted are underlined.  The ranks of the deleted beads in $x$ are $(1, 3, 6, 10, \ldots)$, which are no higher than  the ranks of the deleted beads in $y$, which are $(1, 2, 5, 10, \ldots)$.  After the deletion of these beads, the $k^{\text{th}}$ highest bead in $\pi(x)$ is still as high as the $k^{\text{th}}$ highest bead in $\pi(y)$, and so $\pi(x) \geq_B \pi(y)$.   The converse is illustrated in a similar fashion.


\end{exam}

\bibliographystyle{amsplain}

\bibliography{references}

\end{document}